\documentclass[12pt]{amsart}
 
\usepackage{amssymb, amscd, txfonts}
\usepackage{graphicx}
\usepackage[all]{xy} 
 
\numberwithin{equation}{section}

\newtheorem{theorem}{Theorem}[section]
\newtheorem{proposition}[theorem]{Proposition}
\newtheorem{lemma}[theorem]{Lemma}

\theoremstyle{definition}
\newtheorem{definition}[theorem]{Definition}
\newtheorem{example}[theorem]{Example}

\theoremstyle{remark}
\newtheorem{remark}[theorem]{Remark}

\renewcommand{\hom}{\operatorname{Hom}}
\renewcommand{\ker}{\operatorname{Ker}}

\newcommand{\Z}{\mathbb{Z}}
\newcommand{\Q}{\mathbb{Q}}
\newcommand{\R}{\mathbb{R}}
\newcommand{\C}{\mathbb{C}}
\newcommand{\proj}{{\mathbb P}}

\newcommand{\AG}{A(\Gamma)}
\newcommand{\AGcpt}{A(\Gamma)^{\Sigma}}
\newcommand{\XG}{X^{n}(\Gamma)}
\newcommand{\Sp}{{\rm Sp}(\Lambda)}
\newcommand{\D}{\mathcal{D}}
\newcommand{\volD}{{\rm vol}_{\mathcal{D}}}
\newcommand{\G}{\Gamma}
\newcommand{\X}{\mathcal{X}}
\newcommand{\Xn}{\mathcal{X}^{(n)}}
\newcommand{\LG}{{\rm LG}(\Lambda_{\mathbb{C}})}
\newcommand{\LI}{\Lambda(I)}
\newcommand{\LIQ}{\Lambda(I)_{\mathbb{Q}}}
\newcommand{\LIC}{\Lambda(I)_{\mathbb{C}}}
\newcommand{\GIQ}{\Gamma(I)_{\mathbb{Q}}}
\newcommand{\SpQ}{{\rm Sp}(\Lambda_{\mathbb{Q}})}
\newcommand{\SpLIQ}{{\rm Sp}(\Lambda(I)_{\mathbb{Q}})}
\newcommand{\GLIQ}{{\rm GL}(I_{\mathbb{Q}})}
\newcommand{\WIQ}{W(I)_{\mathbb{Q}}}
\newcommand{\UIQ}{U(I)_{\mathbb{Q}}}
\newcommand{\UIC}{U(I)_{\mathbb{C}}}
\newcommand{\UIR}{U(I)_{\mathbb{R}}}
\newcommand{\VIQ}{V(I)_{\mathbb{Q}}}
\newcommand{\GIZ}{\Gamma(I)_{\mathbb{Z}}}
\newcommand{\WIZ}{W(I)_{\mathbb{Z}}}
\newcommand{\UIZ}{U(I)_{\mathbb{Z}}}
\newcommand{\VIZ}{V(I)_{\mathbb{Z}}}
\newcommand{\GIZbar}{\overline{\Gamma(I)}_{\mathbb{Z}}}
\newcommand{\sym}{{\rm Sym}^{2}}
\newcommand{\XI}{\mathcal{X}_{I}^{(n)}}
\newcommand{\VI}{\mathcal{V}_{I}}
\newcommand{\DI}{\mathcal{D}(I)}
\newcommand{\DLI}{\mathcal{D}_{\Lambda(I)}}
\newcommand{\BI}{\mathcal{B}_{I}}
\newcommand{\TI}{\mathcal{T}_{I}}

\newcommand{\LQ}{\Lambda_{\mathbb{Q}}}

\newcommand{\V}{\mathbb{V}}

\begin{document}

%%%%%%% Title %%%%%%%%%%%%%%%%%%%%%%%%
\title[]{Universal abelian variety and Siegel modular forms}
\author[]{Shouhei Ma}
\thanks{Supported by JSPS KAKENHI 15H05738 and 17K14158.} 
\address{Department~of~Mathematics, Tokyo~Institute~of~Technology, Tokyo 152-8551, Japan}
\email{ma@math.titech.ac.jp}
%\subjclass[2010]{14K10, 11F46}
%\keywords{} 
%\dedicatory{}

\begin{abstract}
We prove that 
the ring of Siegel modular forms of weight divisible by $g+n+1$ is isomorphic to   
the ring of (log) pluricanonical forms on the $n$-fold Kuga family of abelian varieties 
and its certain compactifications, 
for every arithmetic group for a symplectic form of rank $2g>2$. 
We also give applications to the Kodaira dimension of the Kuga variety. 
In most cases, the Kuga variety has canonical singularities. 
\end{abstract} 

\maketitle

%%%%% Introduction %%%%%

\section{Introduction}\label{sec: intro} 

Our purpose in this article is to establish a correspondence between 
Siegel modular forms and pluricanonical forms on the universal family of abelian varieties 
and its compactification, which connects modular forms to the geometry of the universal family. 
Let $\Lambda$ be a free ${\Z}$-module of rank $2g>2$ 
equipped with a nondegenerate symplectic form $\Lambda\times\Lambda\to{\Z}$, 
and $\Gamma$ be a finite-index subgroup of the symplectic group ${\Sp}$ of $\Lambda$. 
Let ${\AG}={\D}/{\G}$ be the Siegel modular variety defined by ${\G}$, 
where ${\D}$ is the Hermitian symmetric domain attached to $\Lambda$. 
Over ${\AG}$ we have the $n$-fold Kuga family ${\XG}\to{\AG}$, 
whose general fibers are 
$n$-fold self products of the abelian varieties or their quotient by $-1$, 
according to whether $-1\not\in{\G}$ or $-1\in{\G}$. 
The space ${\XG}$ is a normal quasi-projective variety of dimension $g(g+2n+1)/2$.  
%Let $H^0({\XG}, K_{{\XG}}^{\otimes m})$ be the space of 
%holomorphic $m$-canonical forms on (the regular locus of) ${\XG}$.  
Let $M_k(\Gamma)$ be the space of Siegel modular forms of weight $k$ with respect to $\Gamma$, 
and $\textit{S}_k(\Gamma)$ be the subspace of cusp forms. 
Our starting point is the following correspondence. 

\begin{theorem}[\S \ref{sec: proof 1.1}]\label{main thm: interior}
We have a natural isomorphism of graded rings 
\begin{equation}\label{main eqn: interior}
\bigoplus_{m\geq0}H^0({\XG}, K_{{\XG}}^{\otimes m}) \simeq 
\bigoplus_{m\geq0} M_{(g+n+1)m}(\Gamma).  
\end{equation}
If $X$ is a smooth projective model of ${\XG}$, 
the $m=1$ component of \eqref{main eqn: interior} induces an isomorphism 
\begin{equation}\label{main eqn: smooth proj cusp form}
H^{0}(X, K_{X}) \simeq \textit{S}_{g+n+1}({\G}). 
\end{equation}
\end{theorem}

The isomorphism \eqref{main eqn: interior} is a consequence of 
natural isomorphisms between the relevant line bundles on $\mathcal{D}$. 
A higher analogue of \eqref{main eqn: interior} is given 
in the form of a Leray spectral sequence that relates vector-valued Siegel modular forms 
to the cohomology of $K_{{\XG}}^{\otimes m}$ (\S \ref{ssec: spectral sequence}). 
The isomorphism \eqref{main eqn: smooth proj cusp form} is a direct extension of 
the result of Hatada \cite{Ha} who considered the case ${\G}<{\rm Sp}(2g, {\Z})$ torsion-free, 
and is a generalization of the result of Shioda \cite{Sh} and Shokurov \cite{Shoku} in the case $g=1$.  

Our main result is an extension of \eqref{main eqn: interior} 
to a certain class of compactification $X$ of ${\XG}$. 
Although our principal interest would be in compact $X$, the result also holds for not fully compact $X$ as well. 

\begin{theorem}[\S \ref{sec: proof 1.2}]\label{main thm: full cpt}
Let $X$ be a complex analytic variety which contains ${\XG}$ as a Zariski open set. 
%Let $\Delta$ be the union of irreducible components of $X-{\XG}$ of codimension $1$. 
Assume that 
\begin{itemize}
\item the singular locus of $X$ has codimension $\geq 2$, 
\item ${\XG}\to{\AG}$ extends to a morphism $X\to {\AGcpt}$ to 
some toroidal compactification ${\AGcpt}$ of ${\AG}$, and  
\item every irreducible component of the boundary divisor 
$\Delta_{X}=X-{\XG}$ of $X$ dominates  
some irreducible component of the boundary divisor 
$\Delta_{A}={\AGcpt}-{\AG}$ of ${\AGcpt}$. 
\end{itemize}
Then the isomorphism \eqref{main eqn: interior} extends to an isomorphism 
\begin{equation}\label{main eqn: cpt}
\bigoplus_{m\geq0}H^0(X, K_{X}^{\otimes m}(m\Delta_{X})) \simeq 
\bigoplus_{m\geq0} M_{(g+n+1)m}(\Gamma).  
\end{equation}
For each $m$, this gives an injection 
\begin{equation}\label{main eqn: cpt cusp form}
\textit{S}_{(g+n+1)m}(\Gamma)\hookrightarrow H^0(K_{X}^{\otimes m}((m-1)\Delta_{X})). 
\end{equation}
If $X$ is compact and the singularities of the pair $(X, (1-m^{-1})\Delta_{X})$ is 
Kawamata log terminal over general points of $\Delta_{A}$, 
then \eqref{main eqn: cpt cusp form} is also surjective. 
\end{theorem}

%The conditions on $X$ would be reasonable if one wants to view $X$ as an extension of the family ${\XG}\to{\AG}$. 
Namikawa \cite{Na} was the first to construct an example of such a compactification $X$ 
for $n=1$ and ${\G}$ the principal congruence subgroups of ${\rm Sp}(2g, {\Z})$ of even level $\geq4$, 
where $X$ is nonsingular and is a projective family over the 2nd Voronoi compactification of ${\AG}$. 
Namikawa constructed his $X$ as a toroidal compactification of $X^{1}({\G})$.  
General theory of toroidal compactification of ${\XG}$ has been then developed in \cite{Pi}, \cite{FC}, \cite{La1}. 
%In a separate paper we will construct an explicit partial compactification of ${\XG}$ 
%over the rank $1$ partial compactification of ${\AG}$ which has canonical singularities when $g\geq4$ 
%and which relates this correspondence with the Fourier-Jacobi development of modular forms. 
A feature of Theorem \ref{main thm: full cpt} is that 
it is obtained without knowing specific geometry of the boundary of $X$. 
%In particular, we find that the ring $\oplus_{m}H^0(K_{X}^{\otimes m}(m\Delta_{X}))$ 
%is invariant for all compactifications $X$ satisfying those conditions.  

The isomorphism \eqref{main eqn: cpt} is derived by showing that 
every $m$-canonical form $\omega$ on ${\XG}$ has at most pole of order $m$ 
along every component of $\Delta_{X}$ (a Koecher type statement). 
We deduce this property by deriving an asymptotic estimate of 
the $L^{2/m}$ norm of $\omega$ around $\Delta_{X}$. 
As a key step, we use the isomorphism \eqref{main eqn: interior} 
to translate the $L^{2/m}$ norm of $\omega$ into the Petersson norm of the corresponding modular form. 
The problem is then reduced to asymptotic estimate of 
the Petersson norm of modular forms around $\Delta_{A}$, 
which is derived by a standard calculation. 

Theorem \ref{main thm: full cpt} generalizes the result of \cite{Ma} in the case $g=1$, 
which in turn is another generalization of the Shioda isomorphism. 
Several new features arise in the case $g>1$, such as 
the Koecher principle and toroidal compactification, 
which make some part of the story different from the case $g=1$. 

We also give applications of Theorems \ref{main thm: interior} and \ref{main thm: full cpt} 
to the Kodaira dimension $\kappa({\XG})$ of ${\XG}$. 
Note that $\kappa({\XG})$ is nondecreasing with respect to $n$ (\cite{Ue}), 
and is bounded by $g(g+1)/2$ (\cite{Ii}). 
In particular, $\kappa({\XG})=g(g+1)/2$ when ${\AG}$ is of general type. 
In general, the isomorphism \eqref{main eqn: smooth proj cusp form} and knowledge about cusp forms 
tell us a bound of $n$ for $\kappa({\XG})\geq 0$ (cf.~Example \ref{ex: g=2-6}). 
Moreover, we have the following. 

\begin{theorem}[\S \ref{sec: Proof 1.3}]\label{main thm: Kodaira dim}
(1) Let $k_{0}$ be a weight such that $\textit{S}_{k_{0}}({\G})$ gives 
a generically finite map ${\AG}\dashrightarrow {\proj}^{N}$. 
Then $\kappa({\XG})=g(g+1)/2$ for $n\geq k_{0}-g-1$. 

(2) Let $X\supset {\XG}$ be a normal compact complex analytic variety 
which satisfies the conditions in Theorem \ref{main thm: full cpt}. 
Then we have 
\begin{equation}\label{main eqn: Kodaira dim}
\kappa(A({\G})^{\Sigma}, \: (g+n+1)L-\Delta_{A}) \: \leq \: 
\kappa(K_{X}) \: \leq \: 
g(g+1)/2,  
\end{equation}
where $L$ is the ${\Q}$-line bundle of modular forms of weight $1$. 
\end{theorem}

Here, for a ${\Q}$-divisor $D$ on a normal compact complex analytic variety $X$, 
we write $\kappa(D)=\kappa(X, D)$ for its Iitaka dimension. 
We have $\kappa(X)\leq \kappa(K_{X})$, where equality holds if $X$ has canonical singularities. 

By (1), $\kappa({\XG})$ stabilizes to $g(g+1)/2$ for large $n$. 
In general, it is not easy to explicitly find a weight $k_{0}$ as in (1), 
and also the resulting bound $n=k_{0}-g-1$ for $\kappa({\XG})=g(g+1)/2$ 
would be far from the actual bound as we are looking only at the canonical map. 
This is improved by (2), which reduces the problem to the study of slope of cusps forms (leaving aside the singularities). 
The behavior of $\kappa(kL-\Delta_{A})$ is relatively more tractable. 
%In particular, if $\textit{S}_{k_{1}}({\G})\ne \{ 0 \}$ for a weight $k_{1}$, 
%then $\kappa(kL-\Delta_{A})=g(g+1)/2$ for all $k>k_{1}$. 
As for the singularities, at least ${\XG}$ has canonical singularities in most cases (\S \ref{sec: singularities}).

%\setcounter{tocdepth}{1}
%\tableofcontents

%\subsection*{Organization of the paper}

This paper is organized as follows. 
\S \ref{sec: basic def} -- \S \ref{sec: toroidal} are recollection of 
Siegel modular varieties and toroidal compactification. 
%For our purpose, we give an explicit description of the Siegel domain realization  
%in a self-contained manner. 
In \S \ref{sec: boundary estimate} we prepare an asymptotic estimate of 
the Petersson norm of local modular forms. 
In \S \ref{sec: L^{2/m} criterion} we prepare a general $L^{2/m}$ criterion for log pluricanonical forms. 
In \S \ref{sec: proof 1.1}, which can be read after \S \ref{sec: basic def}, 
we prove Theorem \ref{main thm: interior}. 
In \S \ref{sec: proof 1.2} we prove Theorem \ref{main thm: full cpt}. 
In \S \ref{sec: Proof 1.3} we prove Theorem \ref{main thm: Kodaira dim}. 
In \S \ref{sec: singularities}, which is independent of other sections,  
we prove that ${\XG}$ has canonical singularities in most cases. 
%The logical relation between these sections is as follows. 
%\begin{equation*}
%\xymatrix{
% &  &  & \S \ref{sec: L^{2/m} criterion} \ar[rd] &  &  \\ 
%\S \ref{sec: basic def} \ar[r] \ar[rd] & \S \ref{sec: Siegel domain} \ar[r]  & 
%\S \ref{sec: toroidal} \ar[r]  & \S \ref{sec: boundary estimate} \ar[r]  & 
%\S \ref{sec: proof 1.2} \ar[r]  & \S \ref{sec: Proof 1.3} \\  
% & \S \ref{sec: proof 1.1} \ar[r] \ar[urrr] & \S \ref{sec: singularities} & & &        
%}
%\end{equation*}

I wish to thank Gavril Farkas for valuable comments 
which led me to study Theorem \ref{main thm: Kodaira dim}.

%%%%%
%%%%% Basic definitions 
%%%%%

\section{Preliminaries}\label{sec: basic def} 

In this section we recall Siegel modular variety, Siegel modular forms, and Kuga family. 
%We define the symmetric domain ${\D}$ inside its compact dual, 
%the Lagrangian Grassmannian, 
%rather than beginning with the Siegel upper half space. 
%Similarly, we define modular forms as invariant sections of the line bundle 
%$\mathcal{O}(-k)$ over ${\D}$. 
%We have chosen this style because it is coordinate-free, 
%so is suitable for dealing with general arithmetic groups ${\G}$. 
%In \S \ref{ssec: factor of automorphy} we explain the translation to the more classical definition. 

%%%Univ marked family 

\subsection{Universal marked family}\label{ssec: marked family}

Let $\Lambda$ be a free abelian group of rank $2g>2$ endowed with 
a nondegenerate symplectic form $(\cdot , \cdot )\colon \Lambda\times\Lambda\to{\Z}$. 
%Let ${\rm G}(g, \Lambda_{{\C}})$ be the Grassmannian 
%parametrizing $g$-dimensional subspaces of $\Lambda_{{\C}}$. 
Let ${\rm LG}(\Lambda_{{\C}}) = {\rm LG}(g, \Lambda_{{\C}})$ 
be the Lagrangian Grassmannian parametrizing 
$g$-dimensional (= maximal) isotropic subspaces. 
%We will often abbreviate ${\LG}={\rm LG}(g, \Lambda_{{\C}})$. 
The Hermitian symmetric domain attached to $\Lambda$ is 
the open subset of ${\rm LG}(\Lambda_{{\C}})$ defined by 
\begin{equation*}
{\D} = \{ \: [V]\in {\rm LG}(\Lambda_{{\C}}) \; | \; i(\cdot, \bar{\cdot})|_{V}>0 \: \}. 
\end{equation*}
Here $i(\cdot, \bar{\cdot})|_{V}>0$ means the condition that 
the Hermitian form $i(\cdot, \bar{\cdot})|_{V}$ on $V$ is positive definite. 
%The condition is also equivalent to 
%$-i(\cdot, \bar{\cdot})$ on $\bar{V}$ being positive-definite. 
This ensures that $\Lambda_{{\C}}=V\oplus \bar{V}$ for $[V]\in{\D}$. 

Let $\underline{\Lambda_{{\C}}}={\D}\times\Lambda_{{\C}}$ be the product vector bundle over ${\D}$, 
$E\to{\D}$ the universal sub bundle of $\underline{\Lambda_{{\C}}}$ 
whose fiber over $[V]\in{\D}$ is $V\subset\Lambda_{{\C}}$, 
and $F=\underline{\Lambda_{{\C}}}/E$ the universal quotient bundle. 
The symplectic pairing defines a canonical isomorphism $F\simeq E^{\vee}$. 
%Via the canonical isomorphism 
%\begin{equation*}
%T{\rm G}(g, \Lambda_{{\C}})|_{{\D}} \simeq \textit{Hom}(E, F) \simeq F\otimes F 
%\end{equation*}
%for the tangent bundle of ${\rm G}(g, \Lambda_{{\C}})$, 
The tangent bundle of ${\D}$ is canonically isomorphic to ${\rm Sym}^{2}F$. 
%\begin{equation}\label{eqn: TD}
%T{\D} \simeq {\rm Sym}^{2}F. 
%\end{equation}

The local system $\underline{\Lambda}={\D}\times\Lambda$ inside $\underline{\Lambda_{{\C}}}$ induces 
a local system of sections of $F\to{\D}$. 
The universal family of abelian varieties over ${\D}$ is defined by 
\begin{equation*}
{\X} = F/\underline{\Lambda} = 
\underline{\Lambda_{{\C}}}/(E+\underline{\Lambda}) 
\simeq E^{\vee}/\underline{\Lambda}. 
\end{equation*}
Here the last isomorphism is defined by the symplectic pairing. 
The fiber of the projection $f\colon {\X}\to{\D}$ over $[V]\in{\D}$ is the abelian variety 
\begin{equation*}
A = \Lambda_{{\C}}/(V+\Lambda)  \simeq V^{\vee}/\Lambda, 
\end{equation*}
polarized by the symplectic form on $\Lambda\simeq H_1(A, {\Z})$. 
%The dual abelian variety is $\bar{V}/\Lambda^{\vee}$. 
We can naturally identify $H^0(\Omega_{A}^{1})=V$. 
Hence if $\Omega_{f}^{1}$ is the relative cotangent bundle of $f$, 
we have a canonical isomorphism 
$f_{\ast}\Omega_{f}^{1} \simeq E$. 

Let $L=\det E$. 
This is restriction of 
the tautological line bundle $\mathcal{O}(-1)$ over ${\mathbb P}(\bigwedge^{g}\Lambda_{{\C}})$ 
by the Pl\"ucker embedding 
${\D} \subset {\LG} \hookrightarrow {\mathbb P}(\bigwedge^{g}\Lambda_{{\C}})$. 
The fiber of $L$ over $[V]\in{\D}$ is $\det V = H^0(K_A)$. 
Hence if $K_{f}=\det \Omega_{f}^{1}$ is the relative canonical bundle of $f$, 
we have a natural isomorphism 
$f_{\ast}K_{f} \simeq L$. 
Since $K_{f}|_{A}\simeq \mathcal{O}_{A}$ for every fiber $A$ of $f$, 
the natural homomorphism 
$f^{\ast}f_{\ast}K_{f} \to K_{f}$ is isomorphic. 
Therefore 
%\begin{equation*}\label{eqn:Kf=fL D}
$K_{f} \simeq f^{\ast}L$. 
%\end{equation*}
On the other hand, 
taking determinant of $\Omega_{{\D}}^{1}\simeq {\sym}E$,  
we also have an ${\rm Sp}(\Lambda_{{\R}})$-equivariant isomorphism  
\begin{equation*}\label{eqn: K=Lg+1 domain}
K_{{\D}} \simeq \det ({\rm Sym}^{2} E) \simeq L^{\otimes g+1}. 
\end{equation*}

For a natural number $n$ we take the $n$-fold self fiber product  
\begin{equation*}
\mathcal{X}^{(n)} = 
{\X}\times_{{\D}} \cdots \times_{{\D}} {\X} \simeq 
F^{\oplus n}/\underline{\Lambda}^{\oplus n}  
\end{equation*}
and let $f_{n}\colon {\Xn}\to{\D}$ be the projection. 
Since $L_{[V]}^{\otimes n} \simeq H^{0}(K_{A})^{\otimes n} \simeq H^{0}(K_{A^{n}})$ 
for $A=V^{\vee}/\Lambda$, 
we have $(f_{n})_{\ast}K_{f_{n}}\simeq L^{\otimes n}$ and 
\begin{equation*}\label{eqn:Kfn=fnLn D}
K_{f_{n}} \simeq f_{n}^{\ast}L^{\otimes n}. 
\end{equation*}

%%%Arithmetic quotient 

\subsection{Quotient by ${\G}$}\label{ssec: quotient}

Let ${\G}$ be a finite-index subgroup of the symplectic group ${\Sp}$ of $\Lambda$. 
%The action of ${\G}$ on $\Lambda_{{\C}}$ induces the action of ${\G}$ on ${\D}$ which is properly discontinuous. 
The quotient space 
%\begin{equation*}
${\AG}={\D}/{\G}$ 
%\end{equation*}
is the Siegel modular variety defined by ${\G}$. 
%This is a normal analytic space of dimension $g(g+1)/2$. 
By Baily-Borel \cite{BB}, ${\AG}$ has the structure of 
a normal quasi-projective variety of dimension $g(g+1)/2$. 

The group ${\G}$ acts on the vector bundle $\underline{\Lambda_{{\C}}}$ equivariantly. 
This preserves $E$ and $\underline{\Lambda}$, 
and thus ${\G}$ acts on ${\Xn}$. 
The quotient space 
\begin{equation*}
{\XG} = {\Xn}/{\G} \simeq F^{\oplus n}/(\Lambda^{\oplus n}\rtimes {\G}) 
\end{equation*}
is called the $n$-fold \textit{Kuga family}. 
This is a normal quasi-projective variety of dimension $g(g+2n+1)/2$ 
fibered over ${\AG}$. 
Here the quasi-projectivity follows from 
Mumford's GIT construction for the case ${\G}={\Sp}$ (\cite{GIT} Chapter 7, \S 2 -- \S 3) 
and Grothendieck's Riemann existence theorem (\cite{Harts} p.~442). 
For $n=1$ and some torsion-free ${\G}$, 
Shimura \cite{Shimura} constructed a projective embedding of $X^{1}({\G})$ 
using theta functions. 
In a special case in $g=2$, its defining equation is determined in \cite{Gu}. 
General members of the fibration 
${\XG}\to{\AG}$ 
%(again denoted by $f^{(n)}$) 
are abelian varieties when $-1\not\in{\G}$, 
and Kummer varieties when $-1\in{\G}$. 
We do not exclude the Kummer case. 

The group ${\G}$ acts on the line bundle $L$ equivariantly. 
A ${\G}$-invariant section of $L^{\otimes k}$ is called 
a \textit{Siegel modular form} of weight $k$ with respect to ${\G}$. 
We write $M_{k}({\G})$ for the space of them. 
We do not need to impose cusp condition by the Koecher principle (see, e.g., \cite{Fre}).

%%%Petersson metric

\subsection{Petersson metric}\label{ssec:Petersson metric}

%We define an ${\rm Sp}(\Lambda_{{\R}})$-invariant Hermitian metric 
%on $L$ and an ${\rm Sp}(\Lambda_{{\R}})$-invariant volume form on ${\D}$, 
%and explain their relationship with canonical forms. 
%via the correspondences \eqref{eqn: fKf=L domain} and \eqref{eqn: K=Lg+1 domain}. 

We fix an isomorphism $\det \Lambda \simeq {\Z}$. 
Let $[V]\in{\D}$. 
For two vectors 
$\omega=v_{1}\wedge \cdots \wedge v_{g}$,  
$\eta=w_{1}\wedge \cdots \wedge w_{g}$ 
of $L_{[V]}=\det V$, the wedge product 
\begin{equation}\label{eqn: wedge}
\omega \wedge \bar{\eta} = 
v_{1}\wedge \cdots \wedge v_{g} \wedge \bar{w}_{1}\wedge \cdots \wedge \bar{w}_{g}
\end{equation}
is a vector of $\det \Lambda_{{\C}}$. 
We define the inner product of $\omega$ and $\eta$ to be 
the image of $i^{g^{2}} \omega \wedge \bar{\eta}$ in $\det \Lambda_{{\C}}\simeq {\C}$. 
This defines a Hermitian metric on the line bundle $L$, 
which is ${\rm Sp}(\Lambda_{{\R}})$-invariant by construction. 
Its $k$-th power defines an ${\rm Sp}(\Lambda_{{\R}})$-invariant Hermitian metric on 
$L^{\otimes k}$ which we denote by $( \: , \: )_{k}$. 
%
%Alternatively, one can also define $( \: , \: )_{1}$ as follows. 
%For $[V]\in{\D}$ we have the natural Hermitian inner product 
%$i(\cdot, \bar{\cdot})|_{V}$ on $V$. 
%This defines an ${\rm Sp}(\Lambda_{{\R}})$-invariant Hermitian metric on the vector bundle $E$. 
%The induced metric on $L=\det E$ is also ${\rm Sp}(\Lambda_{{\R}})$-invariant, 
%so coincides with $( \: , \: )_{1}$ up to a constant. 

Geometrically, $( \: , \: )_{k}$ is the Hodge metric for %the abelian fibration 
$\mathcal{X}^{(k)}\to{\D}$. 

\begin{lemma}\label{lem: weight 1 Petersson geometric}
Let $A=V^{\vee}/\Lambda$ be the abelian variety over $[V]\in{\D}$. 
We identify $\omega, \eta \in L_{[V]}^{\otimes k}$ with 
canonical forms on $A^{k}$ by the natural isomorphism 
$L_{[V]}^{\otimes k} \simeq H^{0}(K_{A^{k}})$. 
Then we have 
\begin{equation*}
(\omega, \eta)_{k} = i^{g^{2}k} \int_{A^{k}}\omega\wedge \bar{\eta}. 
\end{equation*}
%up to a constant independent of $\omega, \eta, V$. 
\end{lemma}

\begin{proof}
When $k=1$, the wedge product \eqref{eqn: wedge} corresponds to 
the $(g, g)$ form $\omega\wedge\bar{\eta}$ on $A$ 
via the isomorphism $\det \Lambda_{{\C}}\simeq H^{2g}(A, {\C})$. 
%If we choose the isomorphism ${\det}\Lambda_{{\Q}}\simeq {\Q}$ so that 
%the positive generator of ${\det}\Lambda$ is mapped to $1\in{\Z}$, 
The isomorphism $\det \Lambda_{{\C}}\simeq {\C}$ coincides with the integration map 
$\int_{A}: H^{2g}(A, {\C}) \to {\C}$. 
When $k>1$, writing $\omega, \eta$ as 
$\omega=\omega_{0}^{\otimes k}$ and $\eta=\eta_{0}^{\otimes k}$ 
with $\omega_{0}, \eta_{0} \in L_{[V]}=H^{0}(K_{A})$, 
we are reduced to the case $k=1$ by iterated integral. 
\end{proof}

We next define an invariant volume form. 
Via the isomorphism $K_{\mathcal{D}}\simeq L^{\otimes g+1}$, 
the metric $( \: , \: )_{g+1}$ induces a Hermitian metric on $K_{\mathcal{D}}$. 
For each $[V]\in{\D}$, we choose a vector $\omega\ne0\in (K_{\mathcal{D}})_{[V]}$ and define 
\begin{equation*}
({\volD})_{[V]} = 
i^{N^{2}}\frac{\omega \wedge \bar{\omega}}{(\omega, \omega)_{g+1}} 
\end{equation*}
where $N = g(g+1)/2$. 
This does not depend on the choice of $\omega$ and defines a volume form on ${\D}$, 
which is ${\rm Sp}(\Lambda_{{\R}})$-invariant by the invariance of $( \: , \: )_{g+1}$. 
If $\omega, \eta$ are two local sections of $K_{\mathcal{D}}$ over some subset of ${\D}$, 
then  
\begin{equation}\label{lem: weight g+1 Petersson geometric}
(\omega, \eta)_{g+1}{\volD} = 
i^{N^{2}}\omega\wedge\bar{\eta}. 
\end{equation}

%\begin{proof}
%It suffices to check this equality at each point. 
%We may assume $\omega, \eta \ne0$. 
%Then $\eta=\alpha\omega$ for some $\alpha\in{\C}$, so we have 
%\begin{equation*}
%\frac{\omega\wedge\bar{\eta}}{(\omega, \eta)_{g+1}} = 
%\frac{\bar{\alpha}\omega\wedge\bar{\omega}}{\bar{\alpha}(\omega, \omega)_{g+1}} = 
%i^{-N^{2}} {\volD}. 
%\end{equation*}
%\end{proof}

%%%Factor of automorphy 

\subsection{Siegel upper half space}\label{ssec: factor of automorphy}

The traditional style defining Siegel modular forms as functions on 
the Siegel upper half space can be realized if we pick up a $0$-dimensional cusp of ${\D}$.  
%which corresponds to a maximal (= rank $g$ primitive) isotropic sublattice of $\Lambda$. 
%We recall this translation. 
%
Let $J$ be a maximal isotropic sublattice of $\Lambda$. 
We choose a maximal isotropic subspace $J'_{{\Q}}$ of $\Lambda_{{\Q}}$ 
such that $\Lambda_{{\Q}}=J_{{\Q}}\oplus J_{{\Q}}'$. 
(The role of $J'_{{\Q}}$ will be auxiliary.) 
We can identify 
$J_{{\Q}}'\simeq J_{{\Q}}^{\vee}$ and 
$(J_{{\Q}}')^{\vee}\simeq J_{{\Q}}$ by the symplectic pairing. 
$J$ determines the Zariski open set 
%$\{  [V]\, | \, V\cap J_{{\C}}= \{ 0\}  \}$  
%of the Grassmannian ${\rm G}(g, \Lambda_{{\C}})$. 
%Via the splitting $\Lambda_{{\Q}}=J_{{\Q}}\oplus J_{{\Q}}'$, 
%this open set is mapped isomorphically to the linear space 
%\begin{equation*}
%{\hom}(J'_{{\C}}, J_{{\C}}) \simeq J_{{\C}}\otimes J_{{\C}} 
%\end{equation*}
%by associating to a linear map $J_{{\C}}'\to J_{{\C}}$ its graph. 
%Taking intersection with ${\LG}$, 
%we obtain the Zariski open set 
\begin{equation*}
H_{J} = \{ \: [V]\in {\LG} \: | \: V\cap J_{{\C}} = \{ 0 \} \: \}
\end{equation*}
of ${\LG}$. 
The choice of $J'_{{\Q}}$ then induces an isomorphism 
$H_{J}\simeq {\rm Sym}^{2}J_{{\C}}$ 
by associating to an element of 
${\rm Sym}^{2}J_{{\C}}\subset J_{{\C}}\otimes J_{{\C}} \simeq {\hom}(J'_{{\C}}, J_{{\C}})$ 
its graph in $J'_{{\C}}\oplus J_{{\C}}=\Lambda_{{\C}}$. 
(Symmetricity corresponds to isotropicity of the graph.) 
The domain ${\D}$ is contained in $H_{J}$, and its image by $H_{J}\to {\rm Sym}^{2}J_{{\C}}$ is 
\begin{equation*}
\mathfrak{H}_{J} = \{ \: \Omega \in {\rm Sym}^{2}J_{{\C}} \: | \: {\rm Im} \, \Omega >0 \: \}. 
\end{equation*}
This is realization of ${\D}$ as a Siegel upper half space. 
If we change $J_{{\Q}}'$, 
the isomorphism $H_{J}\to {\rm Sym}^{2}J_{{\C}}$ is shifted by translation. 

We choose an orientation of $J$. 
This determines a generator of $\bigwedge^{g}J\simeq{\Z}$ which we denote by $\det J$. 
Then we can define a nowhere vanishing section $s_J$ of the line bundle $L$ by the condition 
\begin{equation*}\label{eqn: def sJ}
(s_J([V]), \: \det J) =1, \qquad [V]\in{\D}. 
\end{equation*}
Here $(\: , \: )$ is the paring between 
$L_{[V]}=\det V \subset \bigwedge^{g} \Lambda_{{\C}}$ and 
$\det J \in \bigwedge^{g} \Lambda_{{\C}}$ 
induced from the symplectic form on $\Lambda_{{\C}}$. 
The factor of automorphy associated to the frame $s_J$ is given by 
\begin{equation*}
j(\gamma, [V]) = 
\frac{\gamma( s_J([V]) )}{s_J([\gamma V])} = 
(\gamma (s_J([V])), \: \det J). 
\end{equation*}
Via the trivialization of $L^{\otimes k}$ by $s_J^{\otimes k}$, 
Siegel modular forms of weight $k$ are identified with holomorphic functions $F$ on ${\D}$ satisfying 
\begin{equation*}
F([\gamma V]) = j(\gamma, [V])^k F([V]), \qquad \gamma\in \Gamma, \: \: [V] \in{\D}. 
\end{equation*}

Over $\mathfrak{H}_{J}$, $j(\gamma, [V])$ takes the classical form as follows. 
Choose a basis $l_1, \cdots, l_g$ of $J$ of positive orientation, 
and let $m_1, \cdots, m_g\in J'_{{\Q}}$ be its dual basis, 
namely $(m_i, l_j)=\delta_{ij}$. 
For $[V]\in{\D}$ we can take the basis $\omega_1, \cdots, \omega_g$ of $V$ 
such that $(\omega_i, l_j)=\delta_{ij}$ (normalized basis). 
If we write the matrix expression of $(\omega_1, \cdots, \omega_g)$ 
with respect to $(l_i)_i, (m_j)_j$ in the form ${}^t(\Omega \: I_g)$, 
then $\Omega$ is the symmetric matrix representing 
the image of $[V]$ in ${\rm Sym}^{2}J_{{\C}}$ with respect to $(l_i)_i$. 
Let 
$\begin{pmatrix} A & B \\ C & D \end{pmatrix}$ 
be the matrix representation of $\gamma$ with respect to $(l_i)_i, (m_j)_j$. 
Since 
$s_J([V]) = \omega_1\wedge \cdots \wedge \omega_g$, 
we have 
\begin{equation*}\label{eqn: factor of auto coordinate}
j(\gamma, [V]) 
 = 
(\gamma\omega_1\wedge \cdots \wedge \gamma\omega_g, \: l_1\wedge \cdots \wedge l_g)  
 =  
\det (C\Omega+D).  
\end{equation*}
%This is the classical form of factor of automorphy. 

We also calculate the Petersson metric on $L$ over $\mathfrak{H}_{J}$. 

\begin{lemma}\label{lem: Petersson classical}
Let $\Omega$ be the matrix expression of the image of $[V]\in{\D}$ in $\mathfrak{H}_{J}$. 
Then we have 
\begin{equation*}
(s_{J}([V]),  s_{J}([V]))_{1} = \det ({\rm Im} \, \Omega ) 
\end{equation*}
up to a constant independent of $V$. 
\end{lemma}

\begin{proof}
We use the notation above. 
Since 
$s_{J}([V])=\omega_{1}\wedge \cdots \wedge \omega_{g}$ 
and 
\begin{equation*}
 \omega_{1}\wedge \cdots \wedge \omega_{g} \wedge 
\bar{\omega}_{1} \wedge \cdots \wedge \bar{\omega}_{g}  
 =  
\det \begin{pmatrix} \Omega & \bar{\Omega} \\ I_{g} & I_{g} \end{pmatrix} 
l_{1}\wedge \cdots \wedge l_{g} \wedge 
m_{1} \wedge \cdots \wedge m_{g}, 
\end{equation*}
then $(s_{J}([V]),  s_{J}([V]))_{1}$ equals to a constant multiple of 
\begin{equation*}
\det \begin{pmatrix} \Omega & \bar{\Omega} \\ I_{g} & I_{g} \end{pmatrix} = 
\det \begin{pmatrix} \Omega-\bar{\Omega} & \bar{\Omega} \\ O & I_{g} \end{pmatrix}     
 =  (2i)^{g} \det ({\rm Im} \, \Omega). 
\end{equation*}
\end{proof}

%\begin{corollary}\label{cor: Petersson classical}
%Let $F$ and $G$ be local sections of $L^{\otimes k}$ over some subset of ${\D}$. 
%We identify $F, G$ with functions $F(\Omega), G(\Omega)$ on 
%the corresponding subset of $\mathfrak{H}_{J}$ 
%via the frame $s_{J}^{\otimes k}$ and the isomorphism ${\D}\to\mathfrak{H}_{J}$. 
%Then  
%\begin{equation*}
%(F([V]), G([V]))_{k} = F(\Omega)\cdot \overline{G(\Omega)} \cdot \det ({\rm Im}\, \Omega)^{k}. 
%\end{equation*}
%\end{corollary}

Finally, we express ${\volD}$ 
in terms of the flat volume form on ${\rm Sym}^{2}J_{{\C}}$. 
This recovers the classical form of Petersson inner product. 

\begin{lemma}\label{lem: volD classical}
Let ${\rm vol}_{J}$ be a flat volume form on ${\rm Sym}^{2}J_{{\C}}$. 
Under the isomorphism ${\D}\simeq \mathfrak{H}_{J}$ we have up to a constant 
\begin{equation*}
{\volD} = \frac{1}{\det ({\rm Im} \, \Omega)^{g+1}} {\rm vol}_{J}. 
\end{equation*}
\end{lemma}

\begin{proof}
The canonical form $\omega_{J}$ on ${\D}$ corresponding to the section 
$s_{J}^{\otimes g+1}$ of $L^{\otimes g+1}$  
extends to a translation-invariant canonical form on ${\rm Sym}^{2}J_{{\C}}$. 
%because $s_J$ does so (cf.~Lemma \ref{lem: action U(J)}). 
Hence $\omega_{J}=dz_{1}\wedge \cdots \wedge dz_{N}$ for 
some coordinate $z_{1}, \cdots , z_{N}$ on ${\rm Sym}^{2}J_{{\C}}$. 
Then  
\begin{equation*}
{\volD} 
= i^{N^{2}}\frac{\omega_{J}\wedge\bar{\omega}_{J}}{(\omega_{J}, \omega_{J})_{g+1}}  
= \frac{{\rm vol}_{J}}{(s_{J}, s_{J})_{g+1}} 
= \frac{{\rm vol}_{J}}{\det ({\rm Im} \, \Omega)^{g+1}}. 
\end{equation*}
\end{proof}

%%%%%
%%%%% Siegel domain realization 
%%%%%

\section{Siegel domain realization}\label{sec: Siegel domain} 

In this section we recall the Siegel domain realization (of the third kind) of ${\D}$ 
associated to each cusp. 
We give a self-contained description, following the style of \cite{Lo},  
that is more explicit than the general theory, %as in \cite{AMRT} 
and that does not depend on coordinates so as to be suitable for dealing with general ${\G}$ other than ${\rm Sp}(2g, {\Z})$. 
%This description will be also used in the study of explicit partial compactification of $X^{n}({\G})$. 
%
%This will also be used in a subsequent paper 
%for the study of an explicit partial compactification of $X^{1}({\G})$. 
%

We fix a primitive isotropic sublattice $I$ of $\Lambda$, say of rank $g'$, 
which corresponds to a cusp of ${\D}$. 
We set $g''=g-g'$. 
We write ${\LI}=I^{\perp}/I$, which is a nondegenerate symplectic lattice of rank $2g''$.

%%%Structure of the stabilizer

\subsection{Structure of the stabilizer}\label{ssec: stabilizer}

Let ${\GIQ}$ be the stabilizer of $I_{{\Q}}$ in ${\SpQ}$. 
%We describe the structure of ${\GIQ}$ and ${\GIZ}={\GIQ}\cap \Gamma$. 
% 
%\subsubsection{Over ${\Q}$}\label{sssec: stabilizer Q}
%
We define ${\UIQ}\lhd {\WIQ}\lhd {\GIQ}$ by 
\begin{equation*}
{\UIQ} = {\ker}({\GIQ} \to {\rm GL}(I^{\perp}_{{\Q}})), 
\end{equation*}
\begin{equation*}
{\WIQ} = {\ker}({\GIQ} \to {\SpLIQ}\times{\GLIQ}), 
\end{equation*}
and put ${\VIQ}={\WIQ}/{\UIQ}$.  
The canonical exact sequence 
\begin{equation}\label{eqn: structure G(I)Q}
1 \to {\WIQ} \to {\GIQ} \to {\SpLIQ}\times{\GLIQ} \to 1 
\end{equation}
splits (non-canonically) if we \textit{choose} 
%a lift ${\LIQ}\hookrightarrow I^{\perp}_{{\Q}}$ of ${\LIQ}$ and 
an isotropic subspace $I_{{\Q}}'$ of $\Lambda_{{\Q}}$ 
with $\Lambda_{{\Q}}=I_{{\Q}}^{\perp}\oplus I_{{\Q}}'$   
and let ${\SpLIQ}$ act on $(I_{{\Q}}\oplus I'_{{\Q}})^{\perp} \simeq {\LIQ}$ and 
${\GLIQ}$ act on $I_{{\Q}}\oplus I_{{\Q}}' \simeq I_{{\Q}}\oplus I_{{\Q}}^{\vee}$. 
%This gives a \textit{non-canonical} isomorphism 
%\begin{equation*}\label{eqn: non-cano split of stabilizer Q}
%{\GIQ} \simeq ({\SpLIQ}\times{\GLIQ})\ltimes {\WIQ}. 
%\end{equation*}

Elements of ${\WIQ}$ can be described as follows. 
%\begin{definition}\label{def: transvection}
For $m\in I^{\perp}_{{\Q}}$ and $l\in I_{{\Q}}$ 
we define $T_{m,l}\in {\rm Sp}(\Lambda_{{\Q}})$ by 
\begin{equation*}\label{eqn: def transvection}
T_{m,l}(v) = v + (m, v)l + (l, v)m, \qquad v\in \Lambda_{{\Q}}. 
\end{equation*}
%Since $T_{m,l}$ acts on $I^{\perp}_{{\Q}}$ by $T_{m,l}(v) = v + (m, v)l$,  
Then $T_{m,l} \in {\WIQ}$. 
The following relations hold: 

\begin{enumerate} 
\item $T_{\alpha m, l}=T_{m,\alpha l}$ for $\alpha\in{\Q}$. 

\item $T_{m,l}\circ T_{m,l'}=T_{m,l+l'}$. 

\item $T_{m,l}\circ T_{m',l} = 
%T_{l,2\beta l} \circ T_{m',l} \circ T_{m,l} = 
T_{l,\alpha l}\circ T_{m+m',l}$ 
where 
$\alpha=(m, m')/2$. 

\item $T_{l,l'}=T_{l',l}$ if $l, l'\in I_{{\Q}}$. 
\end{enumerate}

%Since $T_{l,2\beta l}$ acts on $I^{\perp}_{{\Q}}$ trivially, 
%the action of $T_{m,l}$ on $I^{\perp}_{{\Q}}$ is bilinear with respect to $m\in{\LIQ}$ and $l\in I_{{\Q}}$ 
%(but not so on the whole $\Lambda_{{\Q}}$).  

%The equalities (2) and (3) imply the relation 
%\begin{equation*}
%[T_{m,l}, T_{m,l'}]=0, \qquad [T_{m,l}, T_{m',l}]=T_{l,(m,m')l}. 
%\end{equation*}

\begin{lemma}\label{prop: structure WIQ}
(1) The group ${\WIQ}$ is generated by the elements $T_{m,l}$. 
%More specifically, if we take a basis $l_1, \cdots, l_{2g-g'}$ of $I^{\perp}_{{\Q}}$ 
%such that $l_{1}, \cdots, l_{g'}$ span $I_{{\Q}}$, 
%elements of ${\WIQ}$ can be written as compositions of 
%$T_{\alpha_{i,j}l_{i}, l_{j}}$ for some $\alpha_{i,j}\in{\Q}$ 
%where $1\leq i \leq 2g-g'$ and $1\leq j \leq g'$. 

(2) We have the canonical isomorphisms 
\begin{equation*}
{\rm Sym}^{2}I_{{\Q}} \simeq {\UIQ}, \qquad l\cdot l'\mapsto T_{l, l'}, 
\end{equation*}
\begin{equation*}
{\LIQ}\otimes I_{{\Q}} \simeq {\VIQ}, \qquad m\otimes l \mapsto [T_{\tilde{m},l}], 
\end{equation*}
where $\tilde{m}\in I^{\perp}_{{\Q}}$ is a lift of $m\in{\LIQ}$. 
In particular, ${\UIQ}$ and ${\VIQ}$ are ${\Q}$-vector spaces. 
\end{lemma}

\begin{proof}
This can be checked by 
choosing an isotropic subspace $I'_{{\Q}}\subset \Lambda_{{\Q}}$ 
with $\Lambda_{{\Q}}=I^{\perp}_{{\Q}}\oplus I'_{{\Q}}$ 
and calculating the action on $I'_{{\Q}}$ and $(I_{{\Q}}\oplus I_{{\Q}}')^{\perp}$. 
\end{proof}

Thus 
${\WIQ}$ is the unipotent radical of ${\GIQ}$, 
${\UIQ}$ is the center of ${\WIQ}$, and we have the exact sequence 
\begin{equation*}
0\to {\rm Sym}^2I_{{\Q}} \to {\WIQ} \to {\LIQ}\otimes I_{{\Q}} \to 0. 
\end{equation*}
%When $g'=1$, this gives ${\WIQ}$ the structure of a Heisenberg group. 
%
Since ${\UIQ}$ is a normal subgroup of ${\GIQ}$, 
we have the adjoint action of ${\GIQ}$ on ${\UIQ}$. 
Since 
$\gamma \circ T_{m,l} \circ \gamma^{-1} = T_{\gamma m, \gamma l}$ 
for $\gamma\in{\GIQ}$, 
this coincides with the natural action of ${\GIQ}$ on ${\sym}I_{{\Q}}$. 
%(which factorizes through ${\GLIQ}$).  

%We describe the action of ${\UIQ}$ on ${\D}$ in a special case. 
%(The general case will be studied later.) 

%\begin{lemma}\label{lem: action U(J)}
%Suppose that $I$ is maximal, i.e., $g'=g$. 
%We take an isomorphism $H_{I}\simeq {\rm Sym}^{2}I_{{\C}}$ as in \S \ref{ssec: factor of automorphy}.  
%Then the ${\UIQ}$-action on $H_{I}$ coincides with the translation by 
%${\rm Sym}^{2}I_{{\Q}}$ on ${\rm Sym}^{2}I_{{\C}}$. 
%\end{lemma}

%\begin{proof}
%As in \S \ref{ssec: factor of automorphy}, 
%we choose a maximal isotropic subspace $I_{{\Q}}'$ of $\Lambda_{{\Q}}$ with 
%$\Lambda_{{\Q}}=I_{{\Q}}\oplus I'_{{\Q}}$. 
%Take a vector in ${\rm Sym}^{2}I_{{\C}}$ of the form 
%$v\cdot w=v\otimes w +w\otimes v$ where $v, w\in I_{{\C}}$. 
%The corresponding linear map $I'_{{\C}}\to I_{{\C}}$ is 
%\begin{equation*}
%\varphi_{v\cdot w} ( \cdot ) = (v, \cdot ) w + (w, \cdot ) v. 
%\end{equation*}
%Then $T_{l\cdot l'}\in{\UIQ}$ sends the graph of $\varphi_{v\cdot w}$ 
%to the graph of 
%\begin{equation*}
%\varphi_{v\cdot w + l\cdot l'}( \cdot ) = (v, \cdot ) w + (w, \cdot ) v + (l, \cdot ) l' + (l', \cdot ) l. 
%\end{equation*}
%This proves our claim. 
%\end{proof}

%\subsubsection{Over ${\Z}$}

Now let ${\G}$ be a finite-index subgroup of ${\Sp}$ 
and ${\GIZ}={\GIQ}\cap{\G}$ be the stabilizer of $I$ in ${\G}$. 
We put 
\begin{equation*}
{\WIZ} = {\WIQ}\cap{\G},  \quad {\UIZ} = {\UIQ}\cap{\G}, \quad {\VIZ} = {\WIZ}/{\UIZ}.   
\end{equation*} 
Then ${\UIZ}$ is a lattice in ${\UIQ}\simeq {\rm Sym}^{2}I_{{\Q}}$, and 
${\VIZ}$ is a lattice in ${\VIQ}\simeq {\LIQ}\otimes I_{{\Q}}$. 
We also set 
\begin{equation*}
{\GIZbar} = {\GIZ}/{\UIZ},  \qquad  \Gamma_{I} = {\GIZ}/{\WIZ}. 
\end{equation*}
Then $\Gamma_{I}$ is mapped injectively into 
${\rm Sp}(\Lambda(I))\times {\rm GL}(I)$. 
%\begin{equation*}
%\Gamma_{I} = {\GIZ}/{\WIZ} \hookrightarrow {\rm Sp}(\Lambda(I))\times {\rm GL}(I). 
%\end{equation*}
%
By definition we have the canonical exact sequences 
\begin{equation*}\label{eqn: structure of GIZ I}
0 \to {\UIZ} \to {\WIZ} \to {\VIZ} \to 0, 
\end{equation*}
\begin{equation*}\label{eqn: structure of GIZ II}
0 \to {\WIZ} \to {\GIZ} \to \Gamma_{I} \to 0. 
\end{equation*}
%\begin{equation}\label{eqn: structure of GIZbar}
%0 \to {\VIZ} \to {\GIZbar} \to \Gamma_{I} \to 0.  
%\end{equation}
%Although \eqref{eqn: structure G(I)Q} splits, 
%\eqref{eqn: structure of GIZ II} does not necessarily split. 

%%%Siegel domain realization

\subsection{Siegel domain realization}\label{ssec: Siegel domain}

The choice of $I_{{\C}}$ determines the 2-step projection 
%\begin{equation*}\label{eqn: LG projection}
${\LG} \dashrightarrow {\rm LG}(g'', I_{{\C}}^{\perp}) \dashrightarrow {\rm LG}({\LIC})$.  
%\quad 
%V\mapsto W \mapsto {\rm Im}(W\to {\LIC}), 
%\end{equation*}
%where $W=V\cap I^{\perp}_{{\C}}$.  
We shall show that 
restriction of this to ${\D}\subset {\LG}$ 
defines an embedded $2$-step fibration 
\begin{equation}\label{eqn: Siegel domain realization}
{\D} \hookrightarrow {\DI} \hookrightarrow {\rm LG}(\mathcal{K}_{I}) 
\stackrel{\pi_{1}}{\to} {\VI} \stackrel{\pi_{2}}{\to}  {\DLI} 
\end{equation}
where 
\begin{itemize}
\item ${\DLI}$ is the Hermitian symmetric domain attached to $\Lambda(I)$, 
\item ${\VI}\to{\DLI}$ an affine space bundle for a vector bundle, 
\item ${\rm LG}(\mathcal{K}_{I}) \to {\VI}$ a relative Lagrangian Grassmannian, 
\item ${\DI}\to{\VI}$ a principal ${\rm Sym}^{2}I_{{\C}}$-bundle, and 
\item ${\D} \to {\VI}$ a Siegel upper half space bundle. 
\end{itemize}
This is an explicit form of the Siegel domain realization of ${\D}$ at $I$. 
%When $g'=g$, this is the Siegel upper half space model in \S \ref{ssec: factor of automorphy}. 

We define ${\DI}$ and ${\VI}$ by 
\begin{equation*}\label{eqn: def D(I)}
{\DI} = \{ \: [V] \in {\LG} \: | \: i(\cdot, \bar{\cdot})|_{V\cap I_{{\C}}^{\perp}}>0 \: \}, 
\end{equation*}
\begin{equation*}
{\VI} = \{ \: [W] \in {\rm LG}(g'', I^{\perp}_{{\C}})  \: | \:   i(\cdot, \bar{\cdot})|_{W}>0 \: \}. 
\end{equation*}
We consider the linear algebra construction 
\begin{equation}\label{eqn: linear algebra}
V\mapsto (W, V/W) \mapsto W \mapsto {\rm Im}(W \to {\LIC})  
\end{equation}
for $[V]\in {\DI}$ where $W=V\cap I^{\perp}_{{\C}}$. 
Then $[W]\in {\VI}$. 

\begin{lemma}\label{lem: linear algebra}
Let $\Lambda(W) = (W^{\perp}\cap\Lambda_{{\C}})/W$  
and $I_{W}\subset \Lambda(W)$ be the image of $I_{{\C}}$, which is a maximal isotropic subspace of $\Lambda(W)$. 
Then we have $(V/W)\cap I_{W} = \{ 0 \}$ in $\Lambda(W)$. 
Conversely, if $\tilde{V}\cap I_{W} = \{ 0 \}$ for $[\tilde{V}]\in {\rm LG}(\Lambda(W))$, 
then $\tilde{V}=V'/W$ for some $[V']\in {\DI}$ containing $W$.  
\end{lemma}

\begin{proof}
This is straightforward linear algebra. 
\end{proof}

Let $\mathcal{K}_{I}\to{\VI}$ be the symplectic vector bundle 
whose fiber over $[W]\in{\VI}$ is $\Lambda(W)=W^{\perp}/W$, 
and 
${\rm LG}(\mathcal{K}_{I}) = \cup_{[W]} {\rm LG}(\Lambda(W))$ 
be its relative Lagrangian Grassmannian. 
The construction \eqref{eqn: linear algebra} defines the fibration  
\begin{equation*}
{\DI} \hookrightarrow {\rm LG}(\mathcal{K}_{I}) \stackrel{\pi_{1}}{\to} {\VI} \stackrel{\pi_{2}}{\to} {\DLI}. 
\end{equation*}

\begin{proposition}\label{prop: Siegel upper half space bundle}
(1) The image of ${\DI} \hookrightarrow {\rm LG}(\mathcal{K}_{I})$ 
is a principal ${\UIC}$-bundle over ${\VI}$. 
For each $[W]\in {\VI}$, 
$\mathcal{D}\cap \pi_{1}^{-1}([W])$ is a translation of the Siegel upper half space 
$\mathfrak{H}_{I}$ in ${\UIC}\simeq {\rm Sym}^{2}I_{{\C}}$. 

(2) ${\VI}\to{\DLI}$ is an affine space bundle for the vector bundle $F_{I}\otimes I_{{\C}}$ 
where $F_{I}$ is the universal quotient bundle over ${\DLI}$. 
A choice of a lift ${\LIC}\hookrightarrow I^{\perp}_{{\C}}$ of ${\LIC}$ 
determines a section of ${\VI}\to{\DLI}$. 
\end{proposition}

\begin{proof}
(1) By Lemma \ref{lem: linear algebra}, 
${\DI}\cap \pi_{1}^{-1}([W])$ coincides with the Zariski open set  
$\{  [\tilde{V}]  |  \tilde{V}\cap I_{W} = \{ 0\}  \} \simeq {\rm Sym}^{2}I_{W}$ 
of ${\rm LG}(\Lambda(W))$, 
on which ${\UIC}\simeq {\rm Sym}^{2}I_{{\C}}$ acts by translation. 
The second assertion is similar to the case $g'=g$ in \S \ref{ssec: factor of automorphy}. 
%
%For the second assertion, we only need to prove this for one particular $[W]\in{\VI}$, 
%as $\Gamma(I)_{{\R}}$ acts on ${\VI}$ transitively. 
%We take a lift of ${\LIQ}$ and choose $W$ from ${\LIC}\subset I_{{\C}}^{\perp}$. 
%Then $\Lambda(W)\simeq \Lambda(I)_{{\C}}^{\perp}$. 
%Every $[V] \in \pi_{1}^{-1}([W])$ can be written as 
%$V=W\oplus V^{+}$ where $V^{+}=V\cap \Lambda(I)_{{\C}}^{\perp}$. 
%Indeed, it suffices to check that $\dim V' \geq g'$, 
%which in turn follows from the observation that 
%the image of the projection $V\to {\LIC}$ from ${\LIC}^{\perp}$ is contained in $W^{\perp}=W$. 
%Then $i(\cdot, \bar{\cdot})|_{V}>0$ if and only if $i(\cdot, \bar{\cdot})|_{V^{+}}>0$, 
%namely $[V^{+}]\in \mathfrak{H}_{I}$. 

(2) Let $[U]\in {\DLI}$ and $U'\subset I_{{\C}}^{\perp}$ be its inverse image. 
Then %$U'$ is isotropic of dimension $g$,  
%and $i(\cdot, \bar{\cdot})|_{U'}$ is semi positive definite with kernel $I_{{\C}}$. 
%This shows that   
\begin{equation*}
\pi_{2}^{-1}([U]) = \{ \, W\subset U' \: | \: \dim W = g'', \, W\cap I_{{\C}} = \{ 0 \} \, \}. 
\end{equation*}
If we choose a lift $U'\simeq I_{{\C}}\oplus U$ of $U$, 
we obtain an isomorphism 
$\pi_{2}^{-1}([U])\simeq {\hom}(U, I_{{\C}})$ 
by taking the graph of linear maps $U\to I_{{\C}}$. 
%Thus $\pi_{2}^{-1}([U])$ is an affine space for $U^{\vee}\otimes I_{{\C}}$. 
A lift of ${\LIC}$ determines a lift of every $U$ and hence a section of $\pi_{2}$.  
\end{proof}

We have thus obtained a 2-step fibration as in \eqref{eqn: Siegel domain realization}. 
Finally, we describe the action of ${\GIZ}$.  
%Note that ${\GIZ}$ acts on ${\DI}$ naturally. 
%Following the filtration  
%\eqref{eqn: structure of GIZ I}, \eqref{eqn: structure of GIZ II} 
%of ${\GIZ}$, 
We consider in three steps: 
first by ${\UIZ}$, 
then by ${\VIZ}$, 
and finally by $\Gamma_{I}$. 
%The first step was done in Lemma \ref{lem: UIC bundle}, and we consider the rest. 
Let $T_{I}={\UIC}/{\UIZ}$ 
be the algebraic torus associated with the lattice ${\UIZ}$. 
We have 
$\mathfrak{H}_{I}/{\UIZ} = {\rm ord}^{-1}(\mathcal{C}_{I})$ inside $T_{I}$,  
where $\mathcal{C}_{I}\subset {\sym}I_{{\R}}$ is the cone of positive definite forms 
and ${\rm ord}\colon T_{I}\to {\UIR}$ is the projection map as in \cite{AMRT} p.2. 

\begin{proposition}\label{prop: action of GIZ}
(1) 
The quotient ${\TI}={\DI}/{\UIZ}$ is a principal $T_{I}$-bundle over ${\VI}$, 
and ${\BI}={\D}/{\UIZ}$ is a ${\rm ord}^{-1}(\mathcal{C}_{I})$-bundle inside it.  

(2) The group ${\VIZ}$ acts on ${\DLI}$ trivially, 
and on the fibers of ${\VI}\simeq F_I\otimes I_{{\C}}$ 
by translation by the lattice ${\VIZ}$ of ${\LIQ}\otimes I_{{\Q}}$. 
Thus ${\VI}/{\VIZ}$ is a fibration of abelian varieties over ${\DLI}$. 

(3) The group $\Gamma_{I}$ acts on ${\VI}/{\VIZ}\to{\DLI}$ 
by the equivariant action of ${\rm Sp}(\Lambda(I))\times {\rm GL}(I)$ on $F_I\otimes I_{{\C}}$ 
plus translation on the fibers. 
\end{proposition}

\begin{proof}
(1) follows from Proposition \ref{prop: Siegel upper half space bundle}. 
(2) can be seen by expressing points of $\pi_{2}^{-1}([U])$ as graphs of $U\to I_{{\C}}$ 
and calculating the action of $T_{l,m}$ on them. 
%Since ${\VIZ}$ acts on $\Lambda(I)$ trivially, it acts on ${\DLI}$ trivially. 
%Take a point $[U]\in{\DLI}$ and choose $[W]\in \pi_{2}^{-1}([U])$. 
%Recall that $\pi_{2}^{-1}([U])$ is identified with  ${\hom}(W, I_{{\C}})$
%by associating to $\varphi\colon W\to I_{{\C}}$ its graph. 
%Then $T_{m,l}$ 
%acts on $I_{{\C}}^{\perp}$ by $T_{m,l}(v) = v + (m, v)l$,  
%sends the graph of $\varphi$ to the graph of $\varphi+(m, \cdot )l$. 
%Thus $m\otimes l \in {\VIQ}$ acts on $\pi_{2}^{-1}(U)$ by 
%$(m\otimes l)(\varphi)=\varphi+(m, \cdot )l$. 
%This is translation on $U^{\vee}\otimes I_{{\C}}$ by $m\otimes l\in {\VIQ}$ 
%via the pairing map $\Lambda(I)_{{\Q}}\to U^{\vee}$. 
%
(3) can be seen by choosing an isotropic subspace $I'_{{\Q}}$ as before 
and the corresponding sections of \eqref{eqn: structure G(I)Q} and of ${\VI}\to{\DLI}$. 
The translation is given by the difference of the two lifts of an element of $\Gamma_{I}$ 
in ${\GIZ}$ and in the section of \eqref{eqn: structure G(I)Q}. 
%We write $(\Gamma_{I})_{{\Q}}={\SpLIQ}\times{\GLIQ}$. 
%The first assertion is obvious. 
%In order to prove the second, 
%Choose a lift ${\LIQ}\hookrightarrow I_{{\Q}}^{\perp}$ of ${\LIQ}$ 
%and an isotropic subspace $I_{{\Q}}'$ of $\Lambda_{{\Q}}$ 
%with $\Lambda_{{\Q}}=I_{{\Q}}^{\perp}\oplus I_{{\Q}}'$. 
%As explained before, this induces a section 
%$s\colon (\Gamma_{I})_{{\Q}}\hookrightarrow {\GIQ}$ 
%of \eqref{eqn: structure G(I)Q} and a section of ${\VI}\to{\DLI}$. 
%Since $s((\Gamma_{I})_{{\Q}})$ preserves the lifted ${\LIQ}\subset \Lambda_{{\Q}}$, 
%Elements $s(\gamma_{1}, \gamma_{2})$ of $s((\Gamma_{I})_{{\Q}})$ map 
%$\pi_{2}^{-1}([U])\simeq {\hom}(U, I_{{\C}})$ to 
%$\pi_{2}^{-1}([\gamma_{1}U])\simeq {\hom}(\gamma_{1}U, I_{{\C}})$ by 
%$\varphi \mapsto \gamma_{2} \circ \varphi \circ \gamma_{1}^{-1}$. 
%\begin{equation*}
%{\hom}(U, I_{{\C}}) \to {\hom}(\gamma_{1}U, I_{{\C}}), \quad 
%\varphi \mapsto \gamma_{2} \circ \varphi \circ \gamma_{1}^{-1}. 
%\end{equation*}
%Thus $s((\Gamma_{I})_{{\Q}})$ preserves the chosen section of ${\VI}\to{\DLI}$ and 
%acts on ${\VI}$ via the natural equivariant action on $F_{I}\otimes I_{{\C}}$. 
%
%Now let $\gamma\in\Gamma_{I}$ and take its lift $\tilde{\gamma}\in {\GIZ}$.  
%We can write 
%$\tilde{\gamma} = \alpha\cdot s(\gamma)$ 
%for some $\alpha\in{\WIQ}$. 
%The $\gamma$-action on ${\VI}/{\VIZ}$ is the composition of 
%the above action of $s(\gamma)$ %(which preserves the chosen section)  
%and translation by $[\alpha]\in{\VIQ}/{\VIZ}$. 
\end{proof}

%If the exact sequence \eqref{eqn: structure of GIZ II} splits, 
%we can take $\alpha=0$. 
%But it seems that this would not be always the case . 

%%%%%
%%%%% Toroidal compactification 
%%%%%

\section{Toroidal compactification}\label{sec: toroidal}

In this section we recall toroidal compactification 
of ${\AG}$ following \cite{AMRT}, \cite{HKW}. 
We denote by $T(N)=N_{{\C}}/N$ 
the algebraic torus associated to a free ${\Z}$-module $N$ of finite rank. 
We especially write $T_{I}=T({\UIZ})$. 

\subsection{Relative torus embedding}\label{ssec: relative torus embed}

Let $I$ be a primitive isotropic sublattice of $\Lambda$. 
We equip ${\UIR}\simeq {\sym}I_{{\R}}$ with a ${\Z}$-structure by ${\UIZ}$. 
Let $\mathcal{C}_{I} \subset {\UIR}$ be 
the cone of positive definite forms on $I_{{\R}}^{\vee}$,  
and $\mathcal{C}_{I}^{\ast}\subset {\UIR}$ the cone of semi positive definite forms 
whose kernel is defined over ${\Q}$. 
In other words, 
$\mathcal{C}_{I}^{\ast}=\bigcup_{I'}\mathcal{C}_{I'}$ 
where $I'$ ranges over all primitive sublattices of $I$ (including $I'=\{ 0 \}$). 
Recall that $\Gamma(I)_{{\R}}$ acts on ${\UIR}$ by the adjoint action 
which coincides with the natural action on ${\sym}I_{{\R}}$. 
%In particular, it factors through $\Gamma(I)_{{\R}}\to {\rm GL}(I_{{\R}})$. 
%Since ${\rm GL}(I)$ acts on $\mathcal{C}_{I}$ properly discontinuously, 
%so does the image of ${\GIZ}$ in ${\rm GL}(I)$. 
%
A fan $\Sigma=(\sigma_{\alpha})_{\alpha}$ in ${\UIR}$ 
is called ${\GIZ}$-admissible if %the following conditions are satisfied: 
\begin{enumerate}
\item the support of $\Sigma$ is $\mathcal{C}_{I}^{\ast}$, 
\item $\Sigma$ is preserved by the action of ${\GIZ}$, and  
\item $\Sigma/{\GIZ}$ consists of only finitely many cones. 
\end{enumerate}

Let $T_{I}\hookrightarrow T_{I}^{\Sigma}$ be the torus embedding 
defined by the fan $\Sigma$. 
A ray ${\R}_{\geq0}Q$ in $\Sigma$ corresponds to 
a $T_{I}$-orbit of codimension $1$ in the boundary of $T_{I}^{\Sigma}$,  
say $\Delta_{Q}$. 
We always take $Q$ to be a primitive vector of ${\UIZ}$. 
Let $T_{I}^{Q}$ be the torus embedding of $T_{I}$ defined by the ray ${\R}_{\geq0}Q$. 
Then $T_{I}\subset T_{I}^{Q} \subset T_{I}^{\Sigma}$ 
and $\Delta_{Q}$ is the unique boundary divisor of $T_{I}^{Q}$. 
Let $T_{Q}=T({\Z}Q)\simeq {\C}^{\times}$ and 
$\bar{T}_{Q}\simeq {\C}$ be its the standard partial compactification. 
Then $T_{I}^{Q} \simeq (T_{I}\times \bar{T}_{Q})/T_{Q}$ and 
$\Delta_{Q}\simeq T_{I}/T_{Q}\simeq T({\UIZ}/{\Z}Q)$.  
The embedding $T_{Q}\hookrightarrow T_{I}$ extends to 
$\bar{T}_{Q} \hookrightarrow T_{I}^{Q}$
which gives a normal space of $\Delta_{Q}$ at its base point.  

A character $e^{\chi}={\exp}(2\pi i \chi(\cdot))$ of $T_{I}$, 
where $\chi\in U(I)_{{\Z}}^{\vee}$, 
extends holomorphically over $\Delta_{Q}$ 
if and only if $\chi(Q)\geq0$,  
and it is identically $0$ at $\Delta_{Q}$ if and only if $\chi(Q)>0$. 
By restriction, the character group of $\Delta_{Q}$ is identified with 
$Q^{\perp}\cap U(I)_{{\Z}}^{\vee}$. 
If we choose $\chi\in U(I)_{{\Z}}^{\vee}$ such that $\chi(Q)=1$ 
(this is possible because $Q$ is primitive in ${\UIZ}$),  
then $\Delta_{Q}$ is defined by $e^{\chi}=0$ and  
$e^{\chi}$ gives a normal parameter around $\Delta_{Q}$. 
%(Note that $q$ is not necessarily holomorphic over whole $T_{I}^{\Sigma}$.) 

Now let ${\TI}\to{\VI}$ be the principal $T_{I}$-bundle 
constructed in Proposition \ref{prop: action of GIZ}. 
%in which ${\BI}={\D}/{\UIZ}$ is embedded. 
We can form the relative torus embedding 
%\begin{equation*}
$\mathcal{T}_{I}^{\Sigma} %= {\TI}\times_{T_{I}} T_{I}^{\Sigma} 
= ({\TI}\times T_{I}^{\Sigma})/T_{I}$. 
%\end{equation*}
Let $\mathcal{B}_{I}^{\Sigma}$ be the interior of the closure of ${\BI}$ in $\mathcal{T}_{I}^{\Sigma}$. 
This is the partial compactification in the direction of $I$ defined by $\Sigma$. 
Since ${\GIZ}$ preserves $\Sigma$, 
%the ${\GIZbar}$-action on $T_{I}$ extends to $T_{I}^{\Sigma}$. 
%The ${\GIZbar}$-action and the $T_{I}$-action on ${\TI}$ are compatible in the sense that 
%\begin{equation*}
%\gamma(gx) = {\rm Ad}_{\gamma}(g)(\gamma x), \qquad 
%x\in{\TI}, \; g\in T_{I}, \; \gamma\in{\GIZbar}. 
%\end{equation*}
%Thus the ${\GIZbar}$-action on ${\TI}$ extends to $\mathcal{T}_{I}^{\Sigma}$, 
%and so 
the ${\GIZbar}$-action on ${\BI}$ extends to the action on $\mathcal{B}_{I}^{\Sigma}$. 
This is properly discontinuous (cf.~\cite{AMRT}). %and also \S \ref{sec: Siegel domain}). 
%(This can also be seen from the description in \S \ref{sec: Siegel domain}.) 

\subsection{Adjacent cusps}

Let $J$ be a primitive isotropic sublattice of $\Lambda$ that contains $I$. 
The cusp of ${\D}$ associated to $J$ 
is in the closure of the cusp associated to $I$. 
We shall describe the relationship between 
the relative torus embeddings for $I$ and for $J$. 
First note that ${\UIR}\simeq{\sym}I_{{\R}}$ is contained in $U(J)_{{\R}}\simeq{\sym}J_{{\R}}$. 
Then ${\UIZ}$ is a primitive sublattice of $U(J)_{{\Z}}$, 
so $T_{I}$ is a sub torus of $T_{J}$. 
The cone $\mathcal{C}_{I}^{\ast}$ is an extremal sub cone of $\mathcal{C}_{J}^{\ast}$. 
If we have a fan $\Sigma_{J}$ in $U(J)_{{\R}}$ with support $\mathcal{C}_{J}^{\ast}$, 
its restriction $\Sigma_{I}=\Sigma_{J}|_{I}$ to ${\UIR}$ is a fan in ${\UIR}$ 
with support $\mathcal{C}_{I}^{\ast}$. 
Here $\Sigma_{J}|_{I}$ consists of cones $\sigma_{\alpha}$ in $\Sigma_{J}$ with 
$\sigma_{\alpha}\subset \mathcal{C}_{I}^{\ast}$. 
%This is equivalent to the condition that 
%the relative interior of $\sigma_{\alpha}$ has intersection with ${\UIR}$. 
The embedding $T_{I}\hookrightarrow T_{J}$ extends to 
$T_{I}^{\Sigma_{I}}\hookrightarrow T_{J}^{\Sigma_{J}}$. 
%which factorizes as 
%\begin{equation*}
%T_{I}^{\Sigma_{I}}\hookrightarrow 
%T_{J}^{\Sigma_{I}}\simeq  
%T_{J}\times_{T_{I}} T_{I}^{\Sigma_{I}}\hookrightarrow 
%T_{J}^{\Sigma_{J}}. 
%\end{equation*}
%The first map is a closed immersion and 
%the last is an open immersion. 

We set $\overline{U(J)}_{{\Z}}=U(J)_{{\Z}}/{\UIZ}$. 
We have the quotient map 
\begin{equation*}
{\BI}={\D}/{\UIZ} \to \mathcal{B}_{J} = {\D}/U(J)_{{\Z}} 
\end{equation*}
by $\overline{U(J)}_{{\Z}}$. 
Note that 
$U(J)_{{\Z}}\subset {\GIZ}$ and so 
$\overline{U(J)}_{{\Z}}\subset {\GIZbar}$.

\begin{lemma}\label{lem: glue}
For $I\subset J$ 
the quotient map ${\BI}\to\mathcal{B}_{J}$ extends to 
an etale map $\mathcal{B}_{I}^{\Sigma_{I}} \to \mathcal{B}_{J}^{\Sigma_{J}}$. 
More specifically, it factorizes as 
\begin{equation*}
\mathcal{B}_{I}^{\Sigma_{I}} \to  
\mathcal{B}_{I}^{\Sigma_{I}}/\overline{U(J)}_{{\Z}} \hookrightarrow 
\mathcal{B}_{J}^{\Sigma_{J}},  
\end{equation*} 
where the first map is a free quotient map, 
and the second is an open immersion whose image  
does not intersect with the boundary strata of $\mathcal{B}_{J}^{\Sigma_{J}}$ 
corresponding to the cones in $\Sigma_{J}-\Sigma_{I}$. 
\end{lemma}

\begin{proof}
%Since $I\subset J \subset J^{\perp} \subset I^{\perp}$, 
%we have the following commutative diagram between 
%the Siegel domain realizations at $I$ and $J$: 
%\begin{equation*}
%\xymatrix{
%{\D} \ar@{^{(}-_>}[r] \ar@{=}[d] & {\DI} \ar[r] \ar@{^{(}-_>}[d] & 
%{\VI} \ar@{->>}[d] \ar[r] & {\DLI} \ar@{->>}[d] \\ 
%{\D} \ar@{^{(}-_>}[r] & \mathcal{D}(J) \ar[r]  & \mathcal{V}_{J} \ar[r] & \mathcal{D}_{\Lambda(J)}.       
%}
%\end{equation*}
%Here 
%\begin{itemize}
%\item ${\DI}\hookrightarrow \mathcal{D}(J)$ is the natural inclusion, 
%\item ${\VI}\to \mathcal{V}_{J}$ sends $[W]\in{\VI}$ to $[W\cap J_{{\C}}^{\perp}]\in \mathcal{V}_{J}$, and 
%\item ${\DLI}\to \mathcal{D}_{\Lambda(J)}$ sends $[U]\in{\DLI}$ to 
%the image of $U\cap (J/I)_{{\C}}^{\perp}$ in $\Lambda(J)_{{\C}}$. 
%This is projection of ${\DLI}$ to its boundary component. 
%\end{itemize}
Since $I\subset J \subset J^{\perp} \subset I^{\perp}$, 
we have ${\DI} \subset \mathcal{D}(J)$. 
This is clearly compatible with the action of 
${\UIC}\subset U(J)_{{\C}}$. 
Dividing by ${\UIZ}\subset U(J)_{{\Z}}$, we obtain a map 
${\TI}\to \mathcal{T}_{J}$ which is compatible with the action of $T_{I}\hookrightarrow T_{J}$. 
Hence ${\TI}\to \mathcal{T}_{J}$ extends to 
$\mathcal{T}_{I}^{\Sigma_{I}}\to \mathcal{T}_{J}^{\Sigma_{J}}$ 
whose restriction gives 
$\mathcal{B}_{I}^{\Sigma_{I}}\to \mathcal{B}_{J}^{\Sigma_{J}}$. 

We shall observe 
$\mathcal{T}_{I}^{\Sigma_{I}}\to \mathcal{T}_{J}^{\Sigma_{J}}$ 
more closely. 
We put 
\begin{equation*}
\mathcal{T}_{J}^{\Sigma_{I}} = 
(\mathcal{T}_{J}\times T_{J}^{\Sigma_{I}})/T_{J} \simeq 
(\mathcal{T}_{J}\times T_{I}^{\Sigma_{I}})/T_{I}. 
\end{equation*}
Then $\mathcal{T}_{I}^{\Sigma_{I}}\to \mathcal{T}_{J}^{\Sigma_{J}}$ 
factorizes as 
\begin{equation*}
\mathcal{T}_{I}^{\Sigma_{I}}\to 
\mathcal{T}_{J}^{\Sigma_{I}} \to 
\mathcal{T}_{J}^{\Sigma_{J}}. 
\end{equation*}
The second map 
$\mathcal{T}_{J}^{\Sigma_{I}} \to \mathcal{T}_{J}^{\Sigma_{J}}$ 
is an open embedding, 
whose complement consists of boundary strata 
corresponding to the cones in $\Sigma_{J} - \Sigma_{I}$. 
As for the first map 
$\mathcal{T}_{I}^{\Sigma_{I}} \to \mathcal{T}_{J}^{\Sigma_{I}}$, 
note that ${\TI}\to \mathcal{T}_{J}$ is a $T_{I}$-equivariant map between 
the principal $T_{I}$-bundles 
${\TI}\to{\VI}$ and $\mathcal{T}_{J}\to \mathcal{T}_{J}/T_{I}$, 
The map between the bases 
${\VI}\to\mathcal{T}_{J}/T_{I}$ 
factorizes as 
\begin{equation*}
{\VI}= {\DI}/{\UIC} \to {\DI}/{\UIC}+U(J)_{{\Z}} \hookrightarrow \mathcal{D}(J)/{\UIC}+U(J)_{{\Z}} = \mathcal{T}_{J}/T_{I}. 
\end{equation*}
The first map is a free quotient by $\overline{U(J)}_{{\Z}}$, 
and the second is an open embedding. 
Thus $\mathcal{T}_{I}^{\Sigma_{I}} \to \mathcal{T}_{J}^{\Sigma_{I}}$ 
is also a composition of a free quotient by $\overline{U(J)}_{{\Z}}$ 
and an open embedding. 
%factorizes as
%\begin{equation*}
%{\VI}\to {\VI}/\overline{U(J)}_{{\Z}} \to \mathcal{T}_{J}/T_{I}. 
%\end{equation*}
%Lemma \ref{lem: UIC bundle} for $J$ tells us that 
%$\overline{U(J)}_{{\Z}}$ acts on ${\VI}$ freely. 
%The second map 
%${\VI}/\overline{U(J)}_{{\Z}} \to \mathcal{T}_{J}/T_{I}$ 
%is an open embedding because 
%it is the quotient of 
%${\DI}\hookrightarrow \mathcal{D}(J)$ 
%by ${\UIC}+U(J)_{{\Z}}$. 
%Thus ${\VI}\to \mathcal{T}_{J}/T_{I}$ is etale, 
%and so $\mathcal{T}_{I}^{\Sigma_{I}} \to \mathcal{T}_{J}^{\Sigma_{I}}$ is etale too. 
\end{proof}

\subsection{Toroidal compactification}

A toroidal compactification of ${\AG}$ is constructed from the following data. 

\begin{definition}[\cite{AMRT}, \cite{HKW}]
An admissible collection of fans for ${\G}$ is 
a collection $\Sigma=(\Sigma_{I})_{I}$ of fans, 
one for each primitive isotropic sublattice $I$ of $\Lambda$, 
which satisfies the following conditions: 
\begin{enumerate}
\item $\Sigma_{I}$ is a ${\GIZ}$-admissible fan in ${\UIR}$, 
\item $\gamma(\Sigma_{I})=\Sigma_{\gamma I}$ for $\gamma\in{\G}$, and 
\item when $I\subset J$, then $\Sigma_{J}|_{I}=\Sigma_{I}$. 
\end{enumerate}
\end{definition}

We will abbreviate $\mathcal{B}_{I}^{\Sigma_{I}}=\mathcal{B}_{I}^{\Sigma}$ 
when no confusion is likely to occur. 
The toroidal compactification of ${\AG}$ by $\Sigma$ is defined as (\cite{AMRT}, \cite{HKW})
\begin{equation*}
{\AGcpt} = \left( \bigsqcup_{I} \mathcal{B}_{I}^{\Sigma_{I}} \right)/ \sim,  
\end{equation*}
where 
$I$ ranges over all primitive isotropic sublattices of $\Lambda$ 
(including $I=\{0\}$ where ${\BI}=\mathcal{B}_{I}^{\Sigma_{I}}={\D}$), 
and $\sim$ is the equivalence relation generated by the following relations: 
\begin{enumerate}
\item the isomorphism 
$\gamma: \mathcal{B}_{I}^{\Sigma_{I}} \to \mathcal{B}_{\gamma I}^{\Sigma_{\gamma I}}$ 
by $\gamma\in \Gamma$, and  
\item the etale map 
$\mathcal{B}_{I}^{\Sigma_{I}}\to \mathcal{B}_{J}^{\Sigma_{J}}$ 
for $I\subset J$ as in Lemma \ref{lem: glue}. 
\end{enumerate}

Let $\Sigma_{I}^{\circ}=\Sigma_{I}\backslash \cup_{K\subsetneq I}\Sigma_{K}$ 
be the set of cones in $\Sigma_{I}$ whose relative interior is contained in $\mathcal{C}_{I}$. 
We write $\Delta_{\sigma, I}$ for the boundary stratum of 
$\mathcal{B}_{I}^{\Sigma}$ corresponding to a cone $\sigma\in\Sigma_{I}$, 
and let 
%\begin{equation*}
$\Delta_{I} = \bigcup_{\sigma \in \Sigma_{I}^{\circ}} \Delta_{\sigma, I}$ 
%\end{equation*}
be the union of boundary strata that does not come from 
higher dimensional cusps adjacent to $I$. 
By Lemma \ref{lem: glue}, the natural map 
$\Delta_{I}/{\GIZbar} \to {\AGcpt}$ 
is injective. 

\begin{theorem}[\cite{AMRT}]\label{thm: toroidal} 
Let $\Sigma$ be an admissible collection of fans for ${\G}$. 

(1) ${\AGcpt}$ is a compact Moishezon space containing ${\AG}$ as a Zariski open set. 
We have a surjective morphism from 
${\AGcpt}$ to the Satake compactification of ${\AG}$,  
which maps $\Delta_{I}$ to the boundary component for $I$. 

(2) For each primitive isotropic sublattice $I$ of $\Lambda$, the natural map 
\begin{equation*}
\mathcal{B}_{I}^{\Sigma}/{\GIZbar} \to {\AGcpt} 
\end{equation*}
is isomorphic on an open neighborhood of $\Delta_{I}/{\GIZbar}$. 
\end{theorem}

By the property (2) (see \cite{AMRT} p.~175), 
the quotient space $\mathcal{B}_{I}^{\Sigma}/{\GIZbar}$ 
gives a local model of ${\AGcpt}$ around the boundary strata lying over the $I$-cusp. 

%\begin{remark}
%In \cite{HKW} Theorem 3.82, 
%the property (2) is stated for the whole boundary divisor 
%$\mathcal{B}_{I}^{\Sigma_{I}}\backslash {\BI}$ in place of $D_{I}$, 
%but this seems not to hold in general. 
%\end{remark}

%%%Extension of the modular bundle

\subsection{Extension of the modular line bundle}\label{ssec: extend L}

There is a natural number $k'$ such that 
for every $x\in{\D}$ and $\gamma\in{\G}$ with $\gamma(x)=x$, 
$\gamma$ acts trivially on $L_{x}^{\otimes k'}$. 
Then $L^{\otimes k'}$ descends to a line bundle over ${\AG}$. 
In this subsection we extend some multiple of this line bundle over ${\AGcpt}$. 
This is an explicit form of Mumford's extension \cite{Mu}. %which uses singular Hermitian metric. 
We proceed in two steps: 
%\begin{enumerate}
first extend $L$ from ${\BI}$ to $\mathcal{B}_{I}^{\Sigma}$ for each $I$; then 
for some $k$, $L^{\otimes k}$ descends from $\sqcup_{I}\mathcal{B}_{I}^{\Sigma}$ to ${\AGcpt}$. 
%\end{enumerate}

As the first step, let $I$ be a primitive isotropic sublattice of $\Lambda$. 
We choose a maximal isotropic sublattice $J\subset \Lambda$ containing $I$. 
Fix an orientation of $I$ and $J$. 
%This determines an orientation of $J/I$. 
Let $s_{J}$ be the frame of $L$ over ${\D}$ associated to $J$ 
(see \S \ref{ssec: factor of automorphy}).  
Since $s_{J}$ is invariant under ${\UIZ}\subset U(J)_{{\Z}}$, 
it descends to a frame of $L$ over ${\BI}={\D}/{\UIZ}$ 
which we again denote by $s_{J}$. 
Then there exists a unique extension of $L$ to a line bundle over $\mathcal{B}_{I}^{\Sigma}$, 
denoted again by $L$,  such that $s_J$ extends to its frame. 
%We again denote by $L$ the extended line bundle over $\mathcal{B}_{I}^{\Sigma}$. 
%(Specifically, consider the isomorphism $L\simeq \mathcal{O}_{{\BI}}$ induced by $s_J$ 
%and simply extend $\mathcal{O}_{{\BI}}$ to $\mathcal{O}_{\mathcal{B}_{I}^{\Sigma}}$.) 
A section $s$ of $L^{\otimes k}$ over ${\BI}$ 
extends holomorphically over $\mathcal{B}_{I}^{\Sigma}$ 
if and only if the function $s/s_{J}^{\otimes k}$ over ${\BI}$ 
extends holomorphically over $\mathcal{B}_{I}^{\Sigma}$.

\begin{proposition}\label{lem: well-defined extension of L}
The extension of $L$ defined above is independent of the choice of $J$ up to isomorphism. 
Moreover, the equivariant action of ${\GIZbar}$ on $L$ over ${\BI}$ 
extends to an equivariant action over $\mathcal{B}_{I}^{\Sigma}$. 
\end{proposition}

In order to prove this, we consider a decomposition of $s_{J}$. 
Let 
\begin{equation*}
{\D} \hookrightarrow {\rm LG}(\mathcal{K}_{I}) 
\stackrel{\pi_1}{\to} {\VI} \stackrel{\pi_2}{\to}  {\DLI} 
\end{equation*}
be the Siegel domain realization with respect to $I$, 
and $\pi = \pi_{2} \circ \pi_{1}$. 
Let $E_I\to{\DLI}$ be the universal sub bundle over ${\DLI}$ and $L_I=\det E_I$. 
%the modular line bundle over ${\DLI}$. 
We have the frame $s_{J/I}$ of $L_{I}$ 
associated to the oriented, maximal isotropic sublattice $J/I$ of $\Lambda(I)$. 
On the other hand, 
we have the relative tautological bundle $E_{\pi_{1}}$ over ${\rm LG}(\mathcal{K}_{I})$ 
whose fiber over $[V]\in {\D}$ is $V/W$ where $W=V\cap I_{{\C}}^{\perp}$. 
Let $L_{\pi_{1}}=\det E_{\pi_{1}}|_{{\D}}$. 
We have the frame $s_{I}$ of $L_{\pi_{1}}$ defined by the condition 
$(s_{I}([V]), \: \det I)=1$, 
where 
$( \: , \: )$ is the pairing between 
$\det (V/W) \subset \bigwedge^{g'}(\Lambda(W))$ and 
$\det I_{W} \subset \bigwedge^{g'}(\Lambda(W))$ 
induced from the symplectic form on $\Lambda(W)$.  
%where 
%$\Lambda(W)=W^{\perp}/W$. 
%and 
%$I_{W}$ is the image of $I_{{\C}}\to\Lambda(W)$. 
  
\begin{lemma}\label{lem: decomp L by Siegel dom real}
We have a natural isomorphism  
$L\simeq L_{\pi_{1}}\otimes \pi^{\ast}L_I$. 
Under this isomorphism we have 
$s_{J}= s_{I} \otimes \pi^{\ast}s_{J/I}$.  
\end{lemma}

\begin{proof}
By varying the exact sequence of vector spaces 
\begin{equation*}
0 \to V\cap I^{\perp}_{{\C}} \to V \to V/(V\cap I^{\perp}_{{\C}}) \to 0 
\end{equation*}
over ${\D}$, we obtain the exact sequence of vector bundles 
\begin{equation*}\label{eqn: filtration of E by I}
0 \to \pi^{\ast}E_I \to E \to E_{\pi_{1}} \to 0. 
\end{equation*} 
This shows that   
$L\simeq L_{\pi_{1}}\otimes \pi^{\ast}L_I$. 
Since 
\begin{equation*}
(s_{I} \otimes \pi^{\ast}s_{J/I}, \: \det J) = 
(s_{I}, \det I) \cdot (s_{J/I}, \det (J/I))  = 1,  
\end{equation*}
we have $s_{J}= s_{I} \otimes \pi^{\ast}s_{J/I}$. 
\end{proof}

%Geometrically the filtration \eqref{eqn: filtration of E by I} of $E$ 
%is given by the integration of $1$-forms on $V^{\vee}/\Lambda$ 
%along the vanishing cycles $I$.  

\begin{proof}[(Proof of Proposition \ref{lem: well-defined extension of L})] 
If $J'\supset I$ is another maximal isotropic sublattice, 
then 
$s_{J}/s_{J'}=\pi^{\ast}(s_{J/I}/s_{J'/I})$ 
is the pullback of a nowhere vanishing function on ${\DLI}$. 
Since the partial compactification ${\BI}\hookrightarrow \mathcal{B}_{I}^{\Sigma}$ 
is done relatively over ${\VI}$, 
$s_{J}/s_{J'}$ extends to a nowhere vanishing function on $\mathcal{B}_{I}^{\Sigma}$. 
This shows the independence of the extension from $J$. 
If we consider $J'=\gamma J$ for $\gamma\in{\GIZ}$, 
this also implies the second assertion. 
\end{proof}

%There is a natural number $k$ such that 
%for every $x\in \mathcal{B}_{I}^{\Sigma}$ and $\gamma\in{\GIZ}$ with $\gamma(x)=x$, 
%$\gamma$ acts trivially on $L_{x}^{\otimes k}$. 
%Then the extended line bundle $L^{\otimes k}$ over $\mathcal{B}_{I}^{\Sigma}$ 
%descends to a line bundle over $\mathcal{B}_{I}^{\Sigma}/{\GIZbar}$. 

We consider the collection of these extended line bundles over 
the whole $\sqcup_{I}\mathcal{B}_{I}^{\Sigma}$ 
and denote it again by $L$. 
The ${\G}$-action on $L$ over $\sqcup_{I}\mathcal{B}_{I}$ 
extends over $\sqcup_{I}\mathcal{B}_{I}^{\Sigma}$ 
by Proposition \ref{lem: well-defined extension of L}. 
Furthermore, 
if $p\colon \mathcal{B}_{I}^{\Sigma} \to \mathcal{B}_{J}^{\Sigma}$ 
is the etale map for $I\subset J$ as in Lemma \ref{lem: glue}, 
the isomorphism 
$p^{\ast}(L|_{\mathcal{B}_{J}}) \simeq L|_{{\BI}}$ 
over ${\BI}$ 
extends over $\mathcal{B}_{I}^{\Sigma}$, 
because we can use a common frame $s_{K}$ for the extension over both 
$\mathcal{B}_{I}^{\Sigma}$ and $\mathcal{B}_{J}^{\Sigma}$ 
where $K$ is maximal with $K\supset J\supset I$. 

\begin{lemma}[cf.~\cite{Mu}]\label{lemma:Koecher toroidal}
A modular form $F$ of weight $k$, 
as a section of $L^{\otimes k}$ over ${\BI}$, 
extends holomorphically over $\mathcal{B}_{I}^{\Sigma}$. 
$F$ is a cusp form if and only if  
it vanishes at the boundary divisor of $\mathcal{B}_{I}^{\Sigma}$ for all $I$. 
\end{lemma}

\begin{proof}
By the above gluing, we may assume that $I$ is maximal. 
We identify $F$ with a function on $\mathcal{B}_{I}$ via the frame $s_{I}^{\otimes k}$, 
which has Fourier expansion 
\begin{equation*}
F = 
\sum_{\chi \in U(I)_{{\Z}}^{\vee}} a_{\chi} e^{\chi}, \qquad 
e^{\chi} = {\exp}(2\pi i \chi( \cdot )). 
\end{equation*}
%which is the power series expansion by the characters $e^{\chi}$ of the torus $T_{J}$. 
By the Koecher principle, 
we have $a_{\chi}\ne0$ only when 
$\chi$ is semi positive definite. %i.e., in the dual cone of $\mathcal{C}_{I}^{+}$. 
Then $\chi(Q)\geq0$ for every ray ${\R}_{\geq0}Q$ in $\Sigma_{I}$, 
so $e^{\chi}$ extends holomorphically over $\mathcal{B}_{I}^{\Sigma}$ for such $\chi$. 
This proves the first assertion. 

By definition, $F$ is a cusp form if and only if 
$a_{\chi}\ne 0$ only for positive definite $\chi$ at all maximal $I$. 
%Let $U(I)_{{\Z}}^{\vee +}$ be the set of semi positive definite $\chi \in U(I)_{{\Z}}^{\vee}$. 
Since $\chi$ is strictly semi positive definite if and only if 
$\chi \in Q^{\perp}$ for some ray ${\R}_{\geq 0}Q\in \Sigma_{I}$,  
the cuspidal condition is equivalent to 
$a_{\chi}=0$ for all $\chi\in Q^{\perp}$ 
for every ray ${\R}_{\geq 0}Q\in \Sigma_{I}$ at every maximal $I$. 
Since $Q^{\perp}\cap U(I)_{{\Z}}^{\vee}$ is the character group of 
the boundary torus associated to ${\R}_{\geq 0}Q$, 
this is equivalent to the vanishing of $F$ at the boundary of $\mathcal{B}_{I}^{\Sigma}$ 
for every maximal $I$. 
\end{proof}

We choose a natural number $k$ such that 
for every $I$, $x\in \mathcal{B}_{I}^{\Sigma}$, $\gamma\in{\GIZ}$ with $\gamma(x)=x$, 
$\gamma$ acts trivially on $L_{x}^{\otimes k}$. 
Then the line bundle $L^{\otimes k}$ over $\sqcup_{I}\mathcal{B}_{I}^{\Sigma}$ 
descends to a line bundle over 
${\AGcpt}=(\sqcup_{I}\mathcal{B}_{I}^{\Sigma})/\sim$. 
This will be denoted as $L^{\otimes k}$ by abuse of notation 
(for $L$ might not exist as a line bundle over ${\AGcpt}$). 
By Lemma \ref{lemma:Koecher toroidal}, 
we have  
$H^{0}({\AGcpt}, L^{\otimes k})\simeq M_{k}({\G})$.

%%%%%
%%%%% Estimate of Petersson norm 
%%%%%

\section{Asymptotic estimate of Petersson norm}\label{sec: boundary estimate} 

Let ${\AGcpt}$ be a toroidal compactification of ${\AG}$. 
Let $J$ be a maximal isotropic sublattice of $\Lambda$ and  
${\R}_{\geq0}Q$ be a ray in $\Sigma_{J}$. 
%with $Q$ a primitive vector of ${\UIZ}$.  
Let $\Delta_{Q}=\Delta_{Q,J}$ be the corresponding boundary stratum of $\mathcal{B}_{J}^{\Sigma}$. 
The image of $\Delta_{Q}$ in ${\AGcpt}$ is 
a Zariski open set of an irreducible component of the boundary divisor of ${\AGcpt}$. 
In this section we prepare an asymptotic estimate of the Petersson norm of 
a local modular form as the period approaches $\Delta_{Q}$.  

%It is convenient to pass from the $I$-cusp to an adjacent $0$-dimensional cusp.  
%We choose a maximal isotropic sublattice $J\subset \Lambda$ containing $I$. 
%By Lemma \ref{lem: glue}, the projection ${\BI}\to\mathcal{B}_{J}$ extends to 
%an etale map $p\colon \mathcal{B}_{I}^{\Sigma}\to \mathcal{B}_{J}^{\Sigma}$. 
%Then $\Delta_{Q,J}=p(\Delta_{Q,I})$ is 
%the boundary stratum of $\mathcal{B}_{J}^{\Sigma}$ 
%associated to the positive-semidefinite ray ${\R}_{\geq0}Q$ in 
%$U(J)_{{\R}}\simeq {\sym}J_{{\R}}$. 
We choose $\chi\in U(J)_{{\Z}}^{\vee}$ with $\chi(Q)=1$. 
Recall that $q={\exp}(2\pi i\chi( \cdot ))$ gives a normal parameter around $\Delta_{Q}$. 
%(cf.~\S \ref{ssec: relative torus embed}). 
%
We take an arbitrary point $x$ of $\Delta_{Q}$ 
and a small neighborhood $\Delta_{x}$ of $x$ in $\Delta_{Q}$. 
Let $T_{r} \subset \mathcal{B}_{J}^{\Sigma}$ 
be the tubular neighborhood of $\Delta_{x}$ of radius $r$, 
defined by $|q|\leq r$. 
We fix a sufficiently small $0<R \ll 1$ and set  
$W_{\varepsilon} = T_{R}-T_{\varepsilon}$ 
for $0<\varepsilon <R$, 
which is the annulus bundle around $\Delta_{x}$ of radius $[\varepsilon, R]$. 
We want to give an asymptotic estimate of 
\begin{equation*}
\int_{W_{\varepsilon}} (F, F)_{k}^{\beta}{\volD} \qquad (\varepsilon \to 0) 
\end{equation*}
for $F$ a local section of $L^{\otimes k}$ defined around $x$ and $\beta>0$. 
We first compute the asymptotic behavior of the Petersson metric on $L$. 

\begin{lemma}\label{lem: asymptotic Petersson metric}
Let $s_{J}$ be the distinguished frame of $L$ associated to $J$.  
Around each point $x$ of $\Delta_{Q}$, we have  
\begin{equation*}
( s_{J} , s_{J} )_{1} \sim C_{x} \cdot (-\log |q|)^{{\rm rk}(Q)} \qquad (|q|\to 0) 
\end{equation*}
for some constant $C_{x}>0$, where 
${\rm rk}(Q)$ is the rank of $Q$ as a quadratic form. 
\end{lemma}

\begin{proof}
We choose a maximal isotropic subspace $J'_{{\Q}}$ of $\Lambda_{{\Q}}$ 
such that $\Lambda_{{\Q}}=J_{{\Q}}\oplus J'_{{\Q}}$. 
Recall that this induces an isomorphism 
%\begin{equation*}
$\iota \colon {\D} \to \mathfrak{H}_{J}$. %\subset {\sym}J_{{\C}}. 
%\end{equation*}
We identify ${\sym}J_{{\C}}$ with the space of 
$n\times n$ symmetric matrices by taking a basis of $J_{{\Q}}$. 
By Lemma \ref{lem: Petersson classical}, 
if $\Omega=\iota([V])$ for $[V]\in{\D}$, then 
\begin{equation*}
(s_{J}([V]), s_{J}([V]))_{1} = \det ({\rm Im} \, \Omega). 
\end{equation*}

We pick up a point $\Omega_{0}\in \mathfrak{H}_{J}$ 
and consider the flow $\Omega_{t}=\Omega_{0}+itQ$ in $\mathfrak{H}_{J}$, 
where $t\in{\R}_{>0}$. 
The image of $\Omega_{t}$ in $\mathcal{B}_{J}$ converges to a point of $\Delta_{Q}$, 
say $x$, from the normal direction. 
Then 
\begin{equation*}
q = {\exp}(2\pi i \chi(\Omega_{t})) 
= {\exp}(2\pi i\chi(\Omega_{0})) \cdot {\exp}(-2\pi t), 
\end{equation*}
so we have 
\begin{equation*}
t \sim (-2\pi)^{-1} \log |q| \qquad (t\to +\infty). 
\end{equation*}
This shows that 
\begin{equation*}
\det({\rm Im}\Omega_{t}) 
\sim C \cdot t^{{\rm rk}(Q)} 
\sim C \cdot (-\log |q|)^{{\rm rk}(Q)} \qquad 
(t\to +\infty),  
\end{equation*}
where $C$ stands for any unspecified positive constant. 
\end{proof}

Note that if $Q$ belongs to $\mathcal{C}_{I}\subset \mathcal{C}_{J}^{+}$ for $I\subset J$, 
then ${\rm rk}(Q)={\rm rk}(I)$. 
Our main result of this section is the following. 

\begin{proposition}\label{prop: asymptotic Petersson norm}
Let $F$ be a local section of $L^{\otimes k}$ defined over 
a neighborhood of $x\in \Delta_{Q}\subset \mathcal{B}_{J}^{\Sigma}$. 
Let $\beta>0$ be a positive real number. 
Then  
\begin{equation*}
\int_{W_{\varepsilon}} (F, F)_{k}^{\beta}{\volD} = o(\varepsilon^{-\alpha}) 
\qquad (\varepsilon\to 0) 
\end{equation*}
for every $\alpha>0$. 
Moreover, 
when $k\beta \geq g+1$, 
$F$ vanishes at $\Delta_{Q}$ 
if and only if  
\begin{equation*}
\int_{W_{\varepsilon}} (F, F)_{k}^{\beta}{\volD} = O(1) 
\qquad (\varepsilon\to 0). 
\end{equation*}
\end{proposition}

\begin{proof}
Via the frame $s_{J}^{\otimes k}$ 
we identify $F$ with a holomorphic function $F(\Omega)$ defined around $x$. 
Let ${\rm vol}_{J}$ be a flat volume form on ${\sym}J_{{\C}}$. 
By Lemmas \ref{lem: Petersson classical} and \ref{lem: volD classical}, 
we have 
\begin{equation*}
\int_{W_{\varepsilon}} (F, F)_{k}^{\beta}{\volD} = 
\int_{W_{\varepsilon}} 
|F(\Omega)|^{2\beta} \cdot \det ({\rm Im}\, \Omega)^{k\beta - g-1} {\rm vol}_{J}. 
\end{equation*}
Locally around $x$, we can write (up to constant) 
\begin{equation*}
{\rm vol}_{J} 
 = 
d\chi \wedge d\bar{\chi} \wedge {\rm vol}_{\Delta_{x}}  
 =  
|q|^{-2} dq \wedge d\bar{q} \wedge {\rm vol}_{\Delta_{x}}  
 =  
r^{-1} dr \wedge d\theta \wedge {\rm vol}_{\Delta_{x}}
\end{equation*}
for some volume form ${\rm vol}_{\Delta_{x}}$ on $\Delta_{x}\subset \Delta_{Q}$, 
where $q=re^{i\theta}$. 
Therefore 
\begin{equation*}
\int_{W_{\varepsilon}} (F, F)_{k}^{\beta}{\volD} = 
C \cdot \int_{\varepsilon}^{R} r^{-1}dr 
\int_{0}^{2\pi} d\theta 
\int_{\Delta_{x}} 
|F(\Omega)|^{2\beta} \cdot \det ({\rm Im} \, \Omega)^{k\beta-g-1} {\rm vol}_{\Delta_{x}}. 
\end{equation*} 
Since $F(\Omega)=O(1)$ as $r=|q|\to 0$, 
Lemma \ref{lem: asymptotic Petersson metric} implies that 
\begin{eqnarray*}
\int_{W_{\varepsilon}} (F, F)_{k}^{\beta}{\volD} 
& \leq & 
C \cdot \int_{\varepsilon}^{R} r^{-1}dr 
\int_{0}^{2\pi} d\theta 
\int_{\Delta_{x}} |\log r \, |^{g'(k\beta-g-1)} {\rm vol}_{\Delta_{x}} \\ 
& = & 
C \cdot \int_{\varepsilon}^{R} |\log r \, |^{g'(k\beta-g-1)} r^{-1} dr 
\end{eqnarray*}
where $C>0$ are some constants independent of $\varepsilon$ 
and $g'$ is the rank of $Q$. 
We have $\log r=o(r^{-\alpha'})$ for any $\alpha'>0$ as $r\to 0$. 
Hence  
\begin{equation*}
| \log r \, |^{g'(k\beta-g-1)} r^{-1} = o(r^{-1-\alpha}) \qquad (r\to 0)  
\end{equation*}
for any $\alpha>0$. 
It follows that 
\begin{equation*}
\int_{\varepsilon}^{R} | \log r \, |^{g'(k\beta-g-1)} r^{-1} dr  =o(\varepsilon^{-\alpha}) 
\qquad 
(\varepsilon\to 0),  
\end{equation*} 
which proves the first assertion. 

When $F|_{\Delta_Q}\not\equiv 0$, 
this calculation also shows that 
\begin{equation*}
\int_{W_{\varepsilon}} (F, F)_{k}^{\beta}{\volD} 
\: \geq \:   
C' \cdot \int_{\varepsilon}^{R} | \log r \, |^{g'(k\beta-g-1)} r^{-1} dr + ({\rm const}) 
\end{equation*}
for some $C'>0$ independent of $\varepsilon \ll R$. 
When $k\beta\geq g+1$, the right hand side diverges as $\varepsilon \to 0$. 
On the other hand, 
when $F|_{\Delta_Q}\equiv 0$, we have 
$|F(\Omega)|^{2\beta}=O(r^{2\beta})$ and so 
\begin{equation*}
\int_{W_{\varepsilon}} (F, F)_{k}^{\beta}{\volD} 
\; \leq \; 
C \cdot \int_{\varepsilon}^{R} | \log r \, |^{g'(k\beta-g-1)} r^{-1+2\beta} dr  
\; \leq \; 
C \cdot \int_{\varepsilon}^{R} r^{\delta} dr
\end{equation*}
for some $\delta>-1$. 
Therefore 
$\int_{W_{\varepsilon}} (F, F)_{k}^{\beta}{\volD}$ 
converges in this case. 
\end{proof} 

Note that the ``only if'' direction in the second assertion 
holds with no restriction on $k\beta$. 

%\begin{corollary}[cf.~\cite{Fre}]\label{cor: cusp form Pete converge}
%A modular form $F$ of weight $k\geq g+1$ 
%is a cusp form if and only if 
%$\int_{{\AG}}(F, F)_{k}{\volD}<\infty$. 
%\end{corollary}
%
%See \cite{Fre} for the case ${\G}={\rm Sp}(2g, {\Z})$ and $k=g+1$. 

%%%%%
%%%%% L^{2/m} criterion 
%%%%%

\section{$L^{2/m}$ criterion}\label{sec: L^{2/m} criterion} 

This section is independent of the previous sections. 
We prepare a general criterion for the pole order of a pluricanonical form 
in terms of the asymptotic behavior of its integral. 
This will be used in \S \ref{sec: proof 1.2} and \S \ref{sec: Proof 1.3}.

\subsection{$L^{2/m}$ norm of $m$-canonical forms}

Let $U$ be a complex manifold of dimension $N$, 
and $\omega$ a (holomorphic) $m$-canonical form on $U$. 
We define the $L^{2/m}$ norm of $\omega$ as follows. 
Let $\bar{\omega}$ be the complex conjugate of $\omega$. 
After a constant multiple, 
$\omega\wedge\bar{\omega}$ gives a real, nonnegative $C^{\infty}$ section of 
the real line bundle $(\bigwedge^{2N}\Omega_{U, {\R}})^{\otimes m}$, 
where $\Omega_{U, {\R}}$ is the real cotangent bundle of $U$. 
To be more precise, 
if we locally write 
$\omega=f(z)(dz_{1}\wedge \cdots \wedge dz_{N})^{\otimes m}$ 
with $z_{\alpha}=x_{\alpha}+iy_{\alpha}$, then 
\begin{eqnarray*}
\omega \wedge \bar{\omega} 
& = &  
|f(z)|^{2}  
(dz_{1}\wedge \cdots \wedge dz_{N} \wedge 
d\bar{z}_{1}\wedge \cdots \wedge d\bar{z}_{N})^{\otimes m} \\ 
& = & 
i^{N(N-2)m} |f(z)|^{2} 
(dx_{1}\wedge dy_{1} \wedge \cdots \wedge dx_{N}\wedge dy_{N})^{\otimes m}. 
\end{eqnarray*}
Globally, 
if we choose a volume form ${\rm vol}_{U}$ on $U$, 
we can write 
\begin{equation*}
\omega \wedge \bar{\omega} = 
i^{N(N-2)m} \varphi(z) \: {\rm vol}_{U}^{\otimes m} 
\end{equation*}
for some real, nonnegative $C^{\infty}$ function $\varphi(z)$ on $U$. 
We put 
$||\omega ||^{2} = \varphi(z)  {\rm vol}_{U}^{\otimes m}$ 
and define its $m$-th root by 
\begin{equation*}
||\omega||^{2/m} = \sqrt[m]{\varphi(z)} \: {\rm vol}_{U}. 
\end{equation*}
Then $||\omega||^{2/m}$ 
is a real, nonnegative, continuous $(N, N)$ form on $U$ 
which is $C^{\infty}$ outside the zero divisor of $\omega$. 
This definition does not depend on the choice of ${\rm vol}_{U}$. 
The integral 
%\begin{equation*}
$\int_{U} ||\omega||^{2/m}$ 
%\end{equation*}
is the norm we want to look at. 
%When $U\subset {\C}^{n}$ 
%and $\omega=f(z)(dz_{1}\wedge \cdots \wedge dz_{n})^{\otimes m}$ in the coordinate, 
%we can write 
%\begin{equation*}
%\int_{U} ||\omega||^{2/m} = 
%\int_{U} |f(z)|^{2/m} dx_{1}\wedge dy_{1} \wedge \cdots \wedge dx_{n}\wedge dy_{n}. 
%\end{equation*}

\subsection{Criterion for pole order}\label{ssec: L2 criterion}

Now let $X$ be a complex manifold and 
$\Delta \subset X$ be a smooth irreducible divisor. 
We take a normal parameter of $\Delta$ and   
denote by $T_{r}$ the tubular neighborhood of $\Delta$ of radius $r$. 
We fix a sufficiently small $0< R \ll 1$. 
For $0<\varepsilon < R$ we set 
$U_{\varepsilon}=T_{R} - T_{\varepsilon}$, 
which is the annulus bundle of radius $[\varepsilon, R]$ around $\Delta$. 
Let $\omega$ be an $m$-canonical form on $X-\Delta$. 
Our purpose is to relate the pole order of $\omega$ along $\Delta$ 
to the asymptotic behavior of the integral 
\begin{equation*}
\int_{U_{\varepsilon}} || \omega ||^{2/m} \qquad (\varepsilon \to 0). 
\end{equation*}
%relative to $\varepsilon$. 

Since the problem is local, we shall assume that 
$X$ is a polydisc in ${\C}^{N}$, with coordinate $(z_{1}, \cdots , z_{N})$, 
$\Delta$ is defined by $z_{1}=0$, and 
$T_{r}$ is given by $|z_{1}|<r$. 
Then %$X$ is the product of $\Delta\subset {\C}^{n-1}$ and a $z_{1}$-disc,  
$U_{\varepsilon}$ is defined by $\varepsilon\leq |z_{1}| \leq R$. 
We can express 
$\omega=f(z)(dz_{1}\wedge \cdots \wedge dz_{N})^{\otimes m}$ 
for a holomorphic function $f(z)$ on $X-\Delta$. 
Then  
\begin{eqnarray*}
||\omega||^{2/m} 
& = & 
|f(z)|^{2/m} dx_{1}\wedge dy_{1} \wedge \cdots \wedge dx_{N} \wedge dy_{N} \\ 
& = & 
|f(z)|^{2/m} r dr \wedge d\theta \wedge 
dx_{2}\wedge dy_{2} \wedge \cdots \wedge dx_{N} \wedge dy_{N}
\end{eqnarray*}
where 
$z_{\alpha}=x_{\alpha}+iy_{\alpha}$ and 
$z_{1}=re^{i\theta}$. 
So its integral over $U_{\varepsilon}$ is expressed as 
\begin{equation}\label{eqn: integral Ue}
\int_{U_{\varepsilon}} || \omega ||^{2/m} = 
\int_{\varepsilon}^{R} rdr \int_{0}^{2\pi} d\theta 
\int_{\Delta} |f(z)|^{2/m} dx_{2}\wedge \cdots \wedge dy_{N}. 
\end{equation}
The function $f(z)$ is meromorphic over $X$ if and only if we can write 
$f(z)=g(z)/z_{1}^{\nu}$ 
for some $\nu\in{\Z}$ and a holomorphic function $g(z)$ on $X$ 
such that $g|_{\Delta}\not\equiv 0$. 
This $\nu$ is the pole order of $f$ (and of $\omega$) along $\Delta$. 
%If $f(z)$ is not meromorphic over $X$, 
%then in the Laurent expansion of $f$ along $\Delta$ 
%\begin{equation*}
%f(z) = \sum_{m\in{\Z}} \varphi_{m}(z_{2}, \cdots , z_{n}) z_{1}^{m}, 
%\end{equation*}
%we have $\varphi_{m}\not\equiv 0$ for infinitely many $m<0$. 

\begin{proposition}\label{prop: L2 criterion}
Let $\nu$ be the pole order of $\omega$ along $\Delta$. 

(1) We have $\nu \leq m$ if and only if 
$\int_{U_{\varepsilon}}||\omega||^{2/m} = o(\varepsilon^{-2/m})$. 

(2) We have $\nu \leq m-1$ if and only if 
$\int_{U_{\varepsilon}}||\omega||^{2/m} = O(1)$. 
\end{proposition}

\begin{proof}
If $f(z)$ is not meromorphic over $X$, 
then for any $a>0$, 
$|f(z)|^{2}$ diverges faster than $|z_{1}|^{-a}$ 
along an open subset of $\Delta$. 
Then $\int_{U_{\varepsilon}}||\omega||^{2/m}$ 
diverges faster than $\varepsilon^{-b}$ for any $b>0$. 
So we may assume that $f(z)$ is meromorphic over $X$ 
and write 
$f(z)=g(z)/z_{1}^{\nu}$ where 
$g(z)$ is holomorphic over $X$ with $g|_{\Delta}\not\equiv 0$, and 
$\nu$ is the pole order of $\omega$ along $\Delta$. 

By \eqref{eqn: integral Ue}, we have 
%\begin{eqnarray*}
%|| \omega ||^{2/m} 
%& = & 
%r^{-2l/m} |g(z)|^{2/m} dx_{1}\wedge dy_{1}\wedge \cdots \wedge dx_{n}\wedge dy_{n} \\ 
%& = & 
%r^{-2l/m+1} |g(z)|^{2/m} dr \wedge d\theta \wedge 
%dx_{2}\wedge dy_{2}\wedge \cdots \wedge dx_{n}\wedge dy_{n}. 
%\end{eqnarray*}
%Therefore 
\begin{equation*}
\int_{U_{\varepsilon}}||\omega||^{2/m} = 
\int_{\varepsilon}^{R} r^{1-2\nu/m}dr \int_{0}^{2\pi} d\theta 
\int_{\Delta} |g(z)|^{2/m} dx_{2}\wedge \cdots \wedge dy_{N}.  
\end{equation*}  
As a function of $r$, the integral 
\begin{equation*}
\int_{0}^{2\pi} d\theta \int_{\Delta} |g(z)|^{2/m} dx_{2}\wedge \cdots \wedge dy_{N} 
\end{equation*}
is continuous at $0\leq r \leq R$, 
and has a nonzero value at $r=0$ by $g|_{\Delta}\not\equiv 0$. 
Therefore 
\begin{equation*}
\int_{0}^{2\pi} d\theta \int_{\Delta} |g(z)|^{2/m} dx_{2}\wedge \cdots \wedge dy_{N} 
 = 
C+o(1)  
\qquad (r \to 0) 
\end{equation*}
for some constant $C>0$. 
It follows that 
\begin{equation*}
\int_{U_{\varepsilon}}||\omega||^{2/m} 
=  
\int_{\varepsilon}^{R}  r^{1-2\nu /m} (C+o(1)) dr. 
\end{equation*}
When $\nu <m$, this shows that 
$\int_{U_{\varepsilon}}||\omega||^{2/m} = O(1)$. 
When $\nu \geq m$, we obtain 
\begin{equation*}
C' | \log \varepsilon \, | + ({\rm const}) 
\: \leq \: 
\int_{U_{\varepsilon}}||\omega||^{2/m} 
\: \leq \: 
C'' | \log \varepsilon \, | + ({\rm const}) \qquad 
\textrm{when} \: \: \nu =m, 
\end{equation*} 
\begin{equation*}
C' \varepsilon^{2(1-\nu /m)} + ({\rm const}) 
\: \leq \:  
\int_{U_{\varepsilon}}||\omega||^{2/m} 
\: \leq \:  
C'' \varepsilon^{2(1-\nu /m)} + ({\rm const}) \qquad 
\textrm{when} \: \: \nu >m, 
\end{equation*} 
for some constants $C', C''>0$ independent of $\varepsilon \ll 1$. 
This first shows the equivalence in (2). 
The smallest $\nu$ with $\nu >m$ is $\nu =m+1$, for which 
$\varepsilon^{2(1-\nu /m)}=\varepsilon^{-2/m}$. 
This implies the equivalence in (1). 
\end{proof}

For our argument in \S \ref{sec: proof 1.2}, 
it is crucial in (1) 
to pass from the bound $O(\log \varepsilon)$ 
to the (seemingly) weaker $o(\varepsilon^{-2/m})$, 
which creates a room for the estimate. 
 
\begin{remark}\label{remark: independence on normal parameter}
This criterion 
does not depend on the choice of the normal parameter $z_{1}$ for $\Delta$. 
%(nor on the coordinate $z_{2}, \cdots, z_{n}$ of $\Delta$). 
Indeed, if $z_{1}'$ is another normal parameter, 
there exist constants $c, c'>0$ such that 
$c|z_{1}'|\leq |z_{1}|$ and $c'|z_{1}|\leq |z_{1}'|$ 
around $\Delta$. 
If $T_{r}'=\{ |z_{1}'|\leq r \}$ is the tubular neighborhood of radius $r$ 
with respect to $z_{1}'$, then 
$T'_{c'r}\subset T_{r}$ and $T_{cr}\subset T'_{r}$. 
Writing $U_{r}'=T_{R}'-T_{r}'$, 
we have  
$U_{\varepsilon}\subset U'_{c'\varepsilon}$ and 
$U'_{\varepsilon}\subset U_{c\varepsilon}$ 
up to a region independent of $0<\varepsilon \ll 1$. 
Thus, if 
$\int_{U_{\varepsilon}}||\omega||^{2/m}=o(\varepsilon^{-\alpha})$ 
holds, then 
\begin{equation*}
\int_{U'_{\varepsilon}}||\omega||^{2/m} \leq 
\int_{U_{c\varepsilon}}||\omega||^{2/m} + (\textrm{const}) = 
o(c^{-\alpha}\varepsilon^{-\alpha}) = 
o(\varepsilon^{-\alpha}), 
\end{equation*}
and vise versa. 
\end{remark}

We also want to have a simple normal crossing version of (2). 
Let $X$ be again a polydisc in ${\C}^{N}$ 
and let $\Delta$ now be defined by $z_{1}\cdots z_{k}=0$. 
We take a smaller closed polydisc $V\subset X$. 

\begin{proposition}\label{prop: L2 snc}
An $m$-canonical form $\omega$ on $X-\Delta$ 
has at most pole of order $m-1$ along every component of $\Delta$ 
if and only if 
$\int_{V-\Delta} || \omega ||^{2/m}<\infty$. 
\end{proposition}

\begin{proof}
The ``if'' direction follows from Proposition \ref{prop: L2 criterion} (2), 
so we only have to consider the ``only if'' direction. 
We can write 
\begin{equation*}
\omega = 
f(z) \cdot (z_{1}\cdots z_{k})^{1-m} \cdot   
(dz_{1}\wedge \cdots \wedge dz_{N})^{\otimes m} 
\end{equation*}
for some holomorphic function $f(z)$ on $X$. 
Writing 
$z_{\alpha}=r_{\alpha}e^{i\theta_{\alpha}}$, 
we have 
\begin{eqnarray*}
|| \omega ||^{2/m} 
& = & 
C \cdot |f(z)|^{2/m}(r_{1} \cdots r_{k})^{2/m-2} 
dz_{1}\wedge d\bar{z}_{1} \wedge \cdots \wedge dz_{N}\wedge d\bar{z}_{N} \\ 
&  \leq & 
C' \cdot (r_{1} \cdots r_{k})^{2/m-1} 
dr_{1}\wedge d\theta_{1} \wedge \cdots \wedge dr_{k}\wedge d\theta_{k} \wedge 
dz_{k+1}\wedge \cdots \wedge d\bar{z}_{N}. 
\end{eqnarray*}
Since 
$\int^{1}_{\varepsilon}r^{\delta}dr=O(1)$ 
if $\delta>-1$, 
this shows that 
$\int_{V-\Delta} || \omega ||^{2/m}<\infty$. 
\end{proof}

%%%%%
%%%%%  Proof of Theorem 1.1
%%%%%

\section{Proof of Theorem \ref{main thm: interior}}\label{sec: proof 1.1} 

From now on we begin the proof of our main results stated in \S \ref{sec: intro}. 
In this section we prove Theorem \ref{main thm: interior}. 
Most part of this section can be read after \S \ref{sec: basic def}. 
We use the common notation 
$f\colon {\Xn} \to {\D}$ and 
$f\colon {\XG} \to {\AG}$ 
for both projections.

\subsection{Proof of \eqref{main eqn: interior}}\label{ssec: proof 1.1}

We first derive the isomorphism \eqref{main eqn: interior} in Theorem \ref{main thm: interior}. 
Recall from \S \ref{ssec: marked family} that we have the isomorphism 
$K_{f}\simeq f^{\ast}L^{\otimes n}$ over ${\Xn}$  and 
$K_{{\D}}\simeq L^{\otimes g+1}$ over ${\D}$. 
Combining them, we obtain the isomorphism 
\begin{equation}\label{eqn: K=fLg+n+1}
K_{{\Xn}}^{\otimes m} \simeq 
K_{f}^{\otimes m} \otimes f^{\ast}K_{{\D}}^{\otimes m} \simeq 
f^{\ast}L^{\otimes (g+n+1)m} 
\end{equation}
over ${\Xn}$. 
%and 
%$f_{\ast}(K_{{\Xn}}^{\otimes m}) \simeq L^{\otimes (g+n+1)m}$ 
%over ${\D}$. 

Let ${\G}$ be a finite-index subgroup of ${\Sp}$. 
We first consider the case ${\G}$ is torsion-free. 
The ${\G}$-linearized line bundles $K_{{\Xn}}$, $K_{f}$, $K_{{\D}}$ on ${\Xn}$ and ${\D}$ 
descend to the line bundles $K_{{\XG}}$, $K_{f}$, $K_{{\AG}}$ on ${\XG}$ and ${\AG}$. 
Also $L$ descends to a line bundle on ${\AG}$ which we again denote by $L$. 
Since the isomorphism \eqref{eqn: K=fLg+n+1} is ${\G}$-equivariant, 
it descends to the isomorphism 
\begin{equation*}
K_{{\XG}}^{\otimes m} \simeq f^{\ast}L^{\otimes (g+n+1)m} 
\end{equation*}
of line bundles over ${\XG}$. 
%(Or we can repeat the same argument as \eqref{eqn: K=fLg+n+1} for  ${\XG}\to{\AG}$.) 
Taking global sections over ${\XG}$ gives 
\begin{eqnarray*}
H^0({\XG}, K_{{\XG}}^{\otimes m}) 
& \simeq & 
H^0({\XG}, f^{\ast}L^{\otimes (g+n+1)m}) \\ 
& \simeq & 
H^0({\AG}, L^{\otimes (g+n+1)m}) = 
M_{(g+n+1)m}({\G}). 
\end{eqnarray*} 
Clearly this isomorphism is compatible with multiplication, 
and we obtain the isomorphism \eqref{main eqn: interior} in this case.

We next consider the general case ${\G}$ is not necessarily torsion-free. 

\begin{lemma}\label{lem: unramify interior}
The projection ${\Xn}\to{\XG}$ is unramified in codimension $1$. 
\end{lemma}

\begin{proof}
Let $\gamma \ne {\rm id} \in{\G}$ be an element of finite order. 
It suffices to show that the fixed locus of $\gamma$ on ${\Xn}$ has codimension $\geq 2$. 
When $\gamma=-{\rm id}$, the fixed locus is 
the sections of order $\leq2$ points in the abelian fibration ${\Xn}\to{\D}$, 
which has codimension $gn\geq2$. 
When $\gamma \ne \pm{\rm id}$, $\gamma$ acts on ${\D}$ nontrivially. 
If $[V]\in{\D}$ is a fixed point of $\gamma$, 
the $\gamma$-action on the linear space $V$ is nontrivial, 
so $\gamma$ acts on the fiber $(V^{\vee}/\Lambda)^{n}$ nontrivially. 
Then the fixed locus of $\gamma$ on ${\Xn}$ has codimension $\geq 1+1=2$. 
\end{proof}

We choose a torsion-free normal subgroup $\Gamma'\lhd{\G}$ of finite index. 
The quotient group $G={\G}/\Gamma'$ acts on $X^{n}(\Gamma')\to A(\Gamma')$ with quotient ${\XG}\to{\AG}$. 
By the previous step for $\Gamma'$, 
we have an isomorphism 
\begin{equation*}
H^{0}(X^{n}(\Gamma'), K_{X^{n}(\Gamma')}^{\otimes m}) 
\simeq M_{(g+n+1)m}(\Gamma') 
\end{equation*}
which by construction is $G$-equivariant. 
We take the $G$-invariant part of this isomorphism. 
For the right side we have by definition 
\begin{equation*}
M_{(g+n+1)m}(\Gamma')^{G} = M_{(g+n+1)m}(\Gamma). 
\end{equation*}
For the left side, since $X^{n}(\Gamma')\to{\XG}$ is unramified in codimension $1$ 
by Lemma \ref{lem: unramify interior}, we have 
\begin{equation*}
H^{0}(X^{n}(\Gamma'), K_{X^{n}(\Gamma')}^{\otimes m})^{G} 
= H^{0}({\XG}, K_{{\XG}}^{\otimes m}).
\end{equation*}
We thus obtain the isomorphism \eqref{main eqn: interior}. 
\qed 

\begin{example}\label{cor: vanish gk odd}
Let $g, n, m$ be odd and assume $-1\in{\G}$. 
Then there is no nonzero $m$-canonical form on ${\XG}$. 
Indeed, $M_k({\G})=0$ for $k=(g+n+1)m$ odd,  
because $-1\in{\G}$ acts on $L^{\otimes k}$ by multiplication by $(-1)^{gk}=-1$. 
\end{example}

%\begin{example}
%Let ${\G}={\rm Sp}(4, {\Z})$. 
%By Igusa \cite{Ig62}, 
%the ring of modular forms of even weight is free: 
%\begin{equation*}
%\bigoplus_{k\in2{\Z}} M_k({\rm Sp}(4, {\Z})) = 
%{\C}[e_4, e_6, \chi_{10}, \chi_{12}], 
%\end{equation*}
%where $e_k$ is the Eisenstein series of weight $k$, 
%and $\chi_{10}, \chi_{12}$ are the unique cusp forms of weight $10$, $12$ respectively. 
%This implies that 
%$\oplus_m H^0(K_{X}^{\otimes m})$ for $X=X^{1}({\rm Sp}(4, {\Z}))$ is isomorphic to 
%\begin{equation*}
%\bigoplus_{k\in4{\Z}} M_k({\rm Sp}(4, {\Z})) = 
%{\C}[e_4, e_6^2, \chi_{12}, e_6\chi_{10}, \chi_{10}^2] / (e_6^2\cdot \chi_{10}^2 = (e_6\chi_{10})^2).  
%\end{equation*}
%\end{example}

%See, e.g., \cite{Ig62}, \cite{AI}, \cite{FSM}, \cite{Tsuyu} and references there  
%for some cases where the ring of Siegel modular forms is determined. 

\subsection{Proof of \eqref{main eqn: smooth proj cusp form}}

Next we show that the $m=1$ component of \eqref{main eqn: interior} 
gives the isomorphism \eqref{main eqn: smooth proj cusp form}. 
This is a consequence of the equality of the Petersson norm of a modular form 
and the $L^{2}$ norm of the corresponding canonical form. 
For later use in \S \ref{sec: proof 1.2} and \S \ref{sec: Proof 1.3}, we formulate this relation for 
$m$-canonical forms with $m\geq 1$. 
(In this section we only need the case $m=1$.) 

\begin{proposition}\label{prop: L2 criterion modular}
Let $B$ be an open set of ${\D}$. 
Let $F$ be a section of $L^{\otimes (g+n+1)m}$ over $B$ 
and $\omega$ be the $m$-canonical form on $f^{-1}(B)\subset \mathcal{X}^{(n)}$ corresponding to $f^{\ast}F$ 
by the isomorphism $K_{{\Xn}}^{\otimes m}\simeq f^{\ast}L^{\otimes (g+n+1)m}$. 
Then we have 
\begin{equation*}
\int_{f^{-1}(B)} || \omega ||^{2/m} = 
\int_{B} (F, F)_{(g+n+1)m}^{1/m} {\volD}
\end{equation*}
up to a constant independent of $F$ and $B$. 
\end{proposition}

\begin{proof}
Since the problem is local over ${\D}$, 
we may assume that $B$ is sufficiently small. 
Since 
$K_{{\Xn}}^{\otimes m} \simeq f^{\ast}K_{{\D}}^{\otimes m}\otimes K_{f}^{\otimes m}$ 
and $f_{\ast}K_{f}$ is invertible, 
we can write $\omega$ as 
\begin{equation*}
\omega  = 
f^{\ast}\varphi \cdot f^{\ast}\omega_{B}^{\otimes m} \otimes \omega_{f}^{\otimes m}, 
\end{equation*}
where 
$\varphi$ is a holomorphic function on $B$, 
$\omega_{B}$ a nowhere vanishing canonical form on $B$, and 
$\omega_{f}$ a nowhere vanishing relative canonical form on $f^{-1}(B)$. 
Then we have 
\begin{equation*}
|| \omega ||^{2/m} = 
C \cdot f^{\ast}| \varphi |^{2/m} \cdot 
f^{\ast}(\omega_{B}\wedge\bar{\omega}_{B}) \otimes 
(\omega_{f}\wedge \bar{\omega}_{f}), 
\end{equation*}
and hence 
%\begin{equation*}
%|| \omega ||^{2/m} = 
%C \cdot f^{\ast}| \varphi |^{2/m} \cdot 
%f^{\ast}(\omega_{B}\wedge\bar{\omega}_{B}) \otimes 
%(\omega_{f}\wedge \bar{\omega}_{f}). 
%\end{equation*}
%This shows that 
\begin{equation*}
\int_{f^{-1}(B)} || \omega ||^{2/m} = 
C \cdot \int_{B} \left( \int_{f^{-1}(B)/B} \omega_{f}\wedge\bar{\omega}_{f} \right) 
|\varphi|^{2/m} \omega_{B}\wedge\bar{\omega}_{B}, 
\end{equation*}
where $\int_{f^{-1}(B)/B}$ means fiber integral. 

On the other hand, 
under the isomorphism $K_{f}\simeq f^{\ast}L^{\otimes n}$ 
we have $\omega_{f}=f^{\ast}F_{1}$ for some section $F_{1}$ of $L^{\otimes n}$ over $B$,  
and under the isomorphism $K_{{\D}} \simeq L^{\otimes g+1}$  
we have $\omega_{B}=F_{2}$ for some section $F_{2}$ of $L^{\otimes g+1}$ over $B$.  
By construction we have 
$F = \varphi \cdot F_{1}^{\otimes m}\otimes F_{2}^{\otimes m}$, 
and hence 
\begin{equation*}
(F, F)_{(g+n+1)m} = 
|\varphi|^{2} \cdot (F_{1}, F_{1})_{n}^{m} \cdot (F_{2}, F_{2})_{g+1}^{m}. 
\end{equation*}
By Lemma \ref{lem: weight 1 Petersson geometric} and \eqref{lem: weight g+1 Petersson geometric}, 
we see that 
%\begin{equation*}
%(F_{2}, F_{2})_{g+1} {\volD} = C \cdot \omega_{B}\wedge \bar{\omega}_{B}. 
%\end{equation*}
%This shows that  
\begin{eqnarray*}
\int_{B} (F, F)_{(g+n+1)m}^{1/m} {\volD} 
& = & 
\int_{B}|\varphi|^{2/m} \cdot (F_{1}, F_{1})_{n} \cdot (F_{2}, F_{2})_{g+1} {\volD} \\ 
& = & 
\int_{B} \left( \int_{f^{-1}(B)/B} \omega_{f}\wedge \bar{\omega}_{f} \right) 
|\varphi|^{2/m} \omega_{B}\wedge \bar{\omega}_{B} \\ 
& = & 
\int_{f^{-1}(B)} ||\omega ||^{2/m} 
\end{eqnarray*}
up to a constant. 
This proves Proposition \ref{prop: L2 criterion modular}. 
\end{proof}

The isomorphism \eqref{main eqn: smooth proj cusp form} is deduced as follows. 
Let $F\in M_{g+n+1}({\G})$ and $\omega\in H^{0}(K_{{\XG}})$ be the corresponding canonical form. 
By Proposition \ref{prop: L2 criterion modular} with $m=1$, we have 
\begin{equation*}
\int_{{\XG}}\omega \wedge \bar{\omega} = 
\int_{{\AG}}(F, F)_{g+n+1}{\volD} \; \; \leq \infty. 
\end{equation*}
As is well-known, $\int_{{\XG}}\omega \wedge \bar{\omega}$ converges if and only if 
$\omega$ extends holomorphically over a smooth projective model $X$ of ${\XG}$ (cf.~Proposition \ref{prop: L2 criterion}). 
On the other hand, by Proposition \ref{prop: asymptotic Petersson norm}, 
$\int_{{\AG}}(F, F)_{g+n+1}{\volD}$ converges if and only if $F$ is a cusp form. 
Thus $\omega$ extends over $X$ if and only if $F$ is a cusp form. 
\qed

We give application of the correspondence \eqref{main eqn: smooth proj cusp form} to 
the Kodaira dimension of ${\XG}$ in a few cases. 

\begin{example}\label{ex: g=2-6}
For ${\G}=\Gamma_{g}={\rm Sp}(2g, {\Z})$, 
$A_{g,n}=X^{n}(\Gamma_{g})$ is the moduli space of 
$n$-pointed principally polarized abelian varieties of genus $g$. 
When $g\leq 5$, $A_{g,1}$ is unirational (\cite{Ve}, \cite{F-V}). 
On the other hand, when $g\leq 6$, we have the following knowledge about cusp forms. 

(1) 
When $g=2$, by Igusa \cite{Ig62}, we have a unique cusp form of weight $10$, which is the minimal weight. 
Hence $A_{2,7}$ has geometric genus $1$ and $\kappa(A_{2,n})\geq 0$ for $n\geq7$.  

(2) When $g=3$,  by Tsuyumine (\cite{Tsuyu} Corollary in \S 20), 
we have a unique cusp form of weight $12$, which is the minimal weight. 
Hence $A_{3,8}$ has geometric genus $1$ and $\kappa(A_{3,n})\geq 0$ for $n\geq8$.  

(3) When $g=4$, we have the Schottky form as a cusp form of weight $8$ (\cite{Ig82}). 
By Salvati-Manni \cite{SM}, this is the minimal weight and $\dim S_{8}(\Gamma_{4})=1$.  
Hence $A_{4,3}$ has geometric genus $1$ and $\kappa(A_{4,n})\geq 0$ for $n\geq3$. 
$A_{4,2}$ is the only missing case in $g=4$.  

(4) When $g=5$, we have $\dim S_{12}(\Gamma_{5})=2$ by Neeb-Venkov \cite{N-V},  
and this is the minimal weight by Poor-Yuen \cite{P-Y}. 
Hence $A_{5,6}$ has geometric genus $2$ and $\kappa(A_{5,n})>0$ for $n\geq6$. 

(5) When $g=6$, we have $\dim S_{12}(\Gamma_{6})\geq 3$ by Neeb-Venkov \cite{N-V}. 
Hence $A_{6,5}$ has geometric genus $\geq 3$ and $\kappa(A_{6,n})>0$ for $n\geq 5$. 
Poor-Yuen \cite{P-Y} showed that there is no cusp form of weight $\leq 8$. 
\end{example}

\subsection{Vector-valued Siegel modular forms}\label{ssec: spectral sequence}

This subsection is a sort of appendix to \S \ref{ssec: proof 1.1}. 
We combine a generalization of the argument in \S \ref{ssec: proof 1.1} 
with the Leray spectral sequence for ${\XG}\to{\AG}$. 
This gives a spectral sequence that relates the cohomology of $K_{{\XG}}^{\otimes m}$ 
to the cohomology of certain vector bundles on ${\AG}$ 
whose sections are vector-valued modular forms. 
The result of this subsection will not be used in other sections. 

We assume that ${\G}$ is torsion-free. 
The universal quotient bundle $F$ over ${\D}$ descends to a vector bundle over ${\AG}$ 
which we again denote by $F$. 

\begin{proposition}
Assume that ${\G}$ is torsion-free. 
For each $m\geq0$ there exists a spectral sequence 
\begin{equation}\label{eqn: spectral sequence}
E_{2}^{p,q} = H^{p}({\AG}, \bigwedge^{q}(F^{\oplus n})\otimes L^{\otimes(g+n+1)m}) 
\: \:  \Rightarrow \: \: 
H^{p+q}({\XG}, K_{{\XG}}^{\otimes m}). 
\end{equation}
\end{proposition}

\begin{proof}
We abbreviate $X={\XG}$. 
We shall rewrite the $E_{2}$ page of the Leray spectral sequence 
\begin{equation*}
E_{2}^{p,q} = H^{p}({\AG}, R^{q}f_{\ast}(K_{X}^{\otimes m})) 
\: \:  \Rightarrow \: \: 
H^{p+q}(X, K_{X}^{\otimes m}). 
\end{equation*}
Since $K_{X}\simeq f^{\ast}L^{\otimes g+n+1}$, 
we have  
\begin{equation*}
R^{q}f_{\ast}(K_{X}^{\otimes m}) 
 \simeq  
R^{q}f_{\ast}f^{\ast}L^{\otimes (g+n+1)m}
 \simeq  
R^{q}f_{\ast}\mathcal{O}_{X} \otimes L^{\otimes (g+n+1)m} 
\end{equation*}
by the projection formula. 
We shall show that 
\begin{equation*}
R^{q}f_{\ast}\mathcal{O}_{X} \simeq \bigwedge^{q}(F^{\oplus n}). 
\end{equation*}
By Grauert's theorem (\cite{Harts} III.12.9), 
$R^{q}f_{\ast}\mathcal{O}_{X}$ is locally free 
and its fiber over a point $[A]=[V^{\vee}/\Lambda]$ of ${\AG}$ is 
identified with $H^{q}(\mathcal{O}_{A^{n}})$. 
We have the canonical isomorphisms 
\begin{equation*}
H^{q}(\mathcal{O}_{A^{n}}) 
\simeq 
H^{0}(\Omega_{A^{n}}^{q})^{\vee} 
\simeq 
\bigwedge^{q}H^{0}(\Omega_{A^{n}}^{1})^{\vee} 
\simeq 
\bigwedge^{q}(H^{0}(\Omega_{A}^{1})^{\vee})^{\oplus n} 
\simeq  \bigwedge^{q}(V^{\vee})^{\oplus n}  
\end{equation*}
Here the first isomorphism is induced from the Hodge pairing 
\begin{equation*}
H^{q,0}(A^{n}) \times H^{0,q}(A^{n}) \to {\C}, \quad 
(\omega, \eta) \mapsto \int_{A^{n}}\omega\wedge\eta\wedge h^{gn-q}, 
\end{equation*}
where $h$ is the polarization on $A^{n}$ 
induced from the given symplectic form. 
The space $\bigwedge^{q}(V^{\vee})^{\oplus n}$ 
is the fiber of 
$\bigwedge^{q}(E^{\vee})^{\oplus n} \simeq \bigwedge^{q}F^{\oplus n}$ 
over $[A]$. 
\end{proof}

Sections of $\bigwedge^{q}(F^{\oplus n})\otimes L^{\otimes k}$ 
over ${\AG}$ are identified with  
modular forms of weight $k$ for ${\G}$ 
with values in $\bigwedge^{q}({\rm St}^{\oplus n})^{\vee}$ 
where ${\rm St}$ is the standard representation of ${\rm GL}_{g}({\C})$. 
Thus the edge morphism of \eqref{eqn: spectral sequence} 
at the side $p=0$ takes the form 
\begin{equation*}
H^{q}({\XG}, K_{{\XG}}^{\otimes m}) \to 
M_{(g+n+1)m}(\Gamma, \bigwedge^{q}({\rm St}^{\oplus n})^{\vee}). 
\end{equation*}
The edge morphism at the other side $q=0$ takes the form 
\begin{equation*}
H^{p}({\AG}, L^{\otimes (g+n+1)m}) \to 
H^{p}({\XG}, K_{{\XG}}^{\otimes m}). 
\end{equation*}

%\begin{remark}
%Since $\bigwedge^{q}F\otimes L \simeq \bigwedge^{g-q}E$, 
%\eqref{eqn: spectral sequence} also takes the form 
%\begin{equation*}
%E_{2}^{p,q} = H^{p}({\AG}, \bigwedge^{q}E\otimes L^{\otimes(g+n+1)m-1}) 
%\: \:  \Rightarrow \: \: 
%H^{p+q}({\XG}, K_{{\XG}}^{\otimes m}). 
%\end{equation*}
%\end{remark}

%%%%%
%%%%% Proof of Theorem 1.2  
%%%%%

\section{Proof of Theorem \ref{main thm: full cpt}}\label{sec: proof 1.2} 

In this section we prove Theorem \ref{main thm: full cpt}. 
Let us begin with recalling the setting. 
Let $X$ be a complex analytic variety containing ${\XG}$ as a Zariski open set. 
We assume that 
\begin{itemize}
\item the singular locus of $X$ has codimension $\geq2$, 
\item ${\XG}\to{\AG}$ extends to a morphism 
$f\colon X\to {\AGcpt}$ to some toroidal compactification of ${\AG}$, and 
\item every irreducible component of $\Delta_{X}=X-{\XG}$ dominates 
some irreducible component of $\Delta_{A}={\AGcpt}-{\AG}$. 
\end{itemize}
By the second condition, $\Delta_{X}=f^{-1}(\Delta_{A})$ is a divisor of $X$. 
We want to show that the isomorphism 
%\begin{equation}\label{eqn: isom at interior for proof 1.2}
$H^{0}(K_{{\XG}}^{\otimes m}) \simeq M_{(g+n+1)m}({\G})$ 
%\end{equation}
in \S \ref{ssec: proof 1.1} extends to an isomorphism 
\begin{equation*}
H^{0}(X, K_{X}^{\otimes m}(m\Delta_{X})) \simeq M_{(g+n+1)m}({\G}). 
\end{equation*}
Since restriction to ${\XG}\subset X$ gives an inclusion 
\begin{equation*}
H^{0}(X, K_{X}^{\otimes m}(m\Delta_{X})) 
\hookrightarrow 
H^{0}({\XG}, K_{{\XG}}^{\otimes m}), 
\end{equation*}
it is sufficient to show that this is actually equality. 
In other words, we want to show that 
every $m$-canonical form on ${\XG}$ 
has at most pole of order $m$ along every irreducible component of $\Delta_{X}$.  
We deduce this property by applying the $L^{2/m}$ criterion in \S \ref{sec: L^{2/m} criterion}. 
The estimate of $L^{2/m}$ norm required there essentially reduces to  
the asymptotic estimate of Petersson norm in \S \ref{sec: boundary estimate}.

%The intermediate results in this section 
%are formulated in a form that can be applied in \S \ref{sec: Proof 1.3} again. 

\subsection{Pullback to relative torus embedding}

In this subsection we translate the $L^{2/m}$ criterion on $X$ 
to that on the family over the torus fibration associated to each cusp. 
Let $I$ be a primitive isotropic sublattice of $\Lambda$. 
Let 
${\XI} = {\Xn}/{\UIZ}$
be the $n$-fold Kuga family over 
${\BI}={\D}/{\UIZ}$. 
The situation around the $I$-cusp is as follows: 
\begin{equation}\label{eqn: diagram -1}
\xymatrix{
{\XI} \ar[r]^{\hat{p}_{1}} \ar[d]_{f} & {\XI}/{\GIZbar} \ar[r]^{\hat{p}_{2}} \ar[d]_{f} & 
{\XG} \ar@{^{(}-_{>}}[rd] \ar[d]_{f} &   \\ 
{\BI} \ar[r]^{p_{1}} \ar@{^{(}-_{>}}[rd] & {\BI}/{\GIZbar} \ar[r]^{p_{2}} \ar@{^{(}-_{>}}[rd] & 
{\AG} \ar@{^{(}-_{>}}[rd]   &  X \ar[d]_{f}  \\  
 & \mathcal{B}_{I}^{\Sigma} \ar[r]_{p_{1}}  & \mathcal{B}_{I}^{\Sigma}/{\GIZbar} \ar[r]_{p_{2}} & 
{\AGcpt}     
}
\end{equation}
We write 
$p=p_{2}\circ p_{1}$ and $\hat{p}=\hat{p}_{2}\circ \hat{p}_{1}$. 
%By Theorem \ref{thm: toroidal} (2), 
%the morphism 
%$p_{2}:\mathcal{B}_{I}^{\Sigma}/{\GIZbar} \to {\AGcpt}$ 
%is isomorphic around the boundary strata over the $I$-cusp. 

Let ${\R}_{\geq0}Q\subset {\UIR}$ be a positive definite ray in $\Sigma_{I}$, 
where $Q$ is a primitive vector of ${\UIZ}$. 
Let $\Delta_{Q}=\Delta_{Q,I}$ be the boundary stratum of $\mathcal{B}_{I}^{\Sigma}$ 
of codimension $1$ 
corresponding to ${\R}_{\geq0}Q$, 
and we put $\Delta_{Q}'=p(\Delta_{Q})\subset {\AGcpt}$. 
We write $a$ for the ramification index of 
$\mathcal{B}_{I}^{\Sigma}\to {\AGcpt}$ at $\Delta_{Q}$. 
We shall localize the situation. 
We take a general point $x$ of $\Delta_{Q}$, 
its small neighborhood $V$ in $\mathcal{B}_{I}^{\Sigma}$, 
and its small neighborhood $\Delta_{x}$ in $\Delta_{Q}$ contained in $V$. 
Then $y=p(x)$ is a general point of $\Delta_{Q}'$, 
$V'=p(V)$ is a small neighborhood of $y$ in ${\AGcpt}$, 
$\Delta_{y}'=p(\Delta_{x})$ is a small neighborhood of $y$ in $\Delta_{Q}'$, 
and $p\colon \Delta_{x}\to \Delta_{y}'$ is isomorphic. 
In some local coordinates around $x$ and $y$, 
$p\colon V\to V'$ is expressed as 
\begin{equation*}
(z, z_{1}, \cdots , z_{N}) \mapsto (z'=z^{a}, z_{1}, \cdots , z_{N}),  
\end{equation*}
with $\Delta_{Q}$ defined by $z=0$ 
and $\Delta_{Q}'$ defined by $z'=0$. 

\begin{remark}
By Theorem \ref{thm: toroidal} (2), $a$ equals to the ramification index of 
$\mathcal{B}_{I}^{\Sigma}\to \mathcal{B}_{I}^{\Sigma}/{\GIZbar}$ at $\Delta_{Q}$. 
Let $U(I)_{{\Z}}^{\star}={\UIQ}\cap \langle {\G}, -1 \rangle$. 
Using Proposition \ref{prop: action of GIZ},  we can show that 
$a=1$ if $Q/2\notin U(I)_{{\Z}}^{\star}$ and 
$a=2$ if $Q/2\in U(I)_{{\Z}}^{\star}$,  
%\begin{equation*}
%a = 
%\begin{cases} 
%1 & Q/2\notin U(I)_{{\Z}}^{\star}, \\ 
%2 & Q/2\in U(I)_{{\Z}}^{\star}, 
%\end{cases}
%\end{equation*}
like the classification of regular/irregular cusps in the case $g=1$. 
We do not need this precise information for the proof of Theorem \ref{main thm: full cpt}. 
\end{remark}

We set 
$U'=f^{-1}(V')\subset X$, 
$V^{\circ}=V\backslash \Delta_{Q}\subset {\BI}$ and 
$U^{\circ}=f^{-1}(V^{\circ})\subset {\XI}$. 
The situation is  
\begin{equation}\label{eqn: diagram -2}
\xymatrix{
U^{\circ} \ar@{->>}[r]^{\hat{p}} \ar[d]_{f} & U'\cap {\XG} \ar@{^{(}-_{>}}[rd] \ar[d]_{f} &   \\ 
V^{\circ}  \ar@{->>}[r]^{p} \ar@{^{(}-_{>}}[rd] & 
V'\backslash \Delta_{Q}' \ar@{^{(}-_{>}}[rd]   &  U' \ar[d]_{f}  \\  
  & V \ar@{->>}[r]_{p} & V'      
}
\end{equation}
which is a localization of \eqref{eqn: diagram -1}.  
Here $U^{\circ}\to V^{\circ}$ is an abelian fibration, 
$U'\cap {\XG} \to V'\backslash \Delta_{Q}'$ is an abelian or Kummer fibration, 
$p$ has degree $a$, 
and $\hat{p}$ has degree $a$ or $2a$. 

Let $T_{r}$ be the tubular neighborhood of $\Delta_{x}$ of radius $|z|=r$.  
Then $p(T_{r})$ is the tubular neighborhood of $\Delta_{y}'$ 
of radius $|z'|=r^{a}$. 
We fix a sufficiently small $0< R \ll 1$, 
and for $0 < \varepsilon < R$ we set 
\begin{equation*}
V_{\varepsilon} = T_{R} - T_{\varepsilon} \subset V^{\circ}, 
\end{equation*}
\begin{equation*}
V_{\varepsilon}' = p(V_{\varepsilon}) \subset V'\backslash \Delta_{Q}'. 
\end{equation*}
Then $V_{\varepsilon}$ is the annulus bundle of radius 
$\varepsilon \leq |z| \leq R$ 
around $\Delta_{x}$, 
and $V_{\varepsilon}'$ is the annulus bundle of radius 
$\varepsilon^{a}\leq |z'| \leq R^{a}$ 
around $\Delta_{y}'$. 
Let 
\begin{equation*}
U_{\varepsilon} = f^{-1}(V_{\varepsilon}) \subset U^{\circ}, 
\end{equation*}
\begin{equation*}
U_{\varepsilon}' = f^{-1}(V_{\varepsilon}') \subset U' \cap {\XG}, 
\end{equation*}
be the families over these bases. 
We have as restriction of \eqref{eqn: diagram -2}
\begin{equation*}\label{eqn: diagram -3}
\xymatrix{
U_{\varepsilon} \ar@{->>}[r]^{\hat{p}} \ar[d]_{f} & U_{\varepsilon}' \ar[d]^{f}   \\ 
V_{\varepsilon} \ar@{->>}[r]^{p} & V_{\varepsilon}'  
}
\end{equation*}
Then $U_{\varepsilon}'$ gives an annulus bundle 
around general point of each component of 
$\Delta_{X}$ lying over $\Delta_{Q}'$. 
To be more precise, 
let $f^{-1}(\Delta_{Q}')=\sum_{i}\Delta_{i}$ be the irreducible decomposition of 
the reduced divisor $f^{-1}(\Delta_{Q}') \subset X$. 
We can write 
%\begin{equation*}
$f^{\ast}\Delta_{Q}'=\sum_{i}d_{i}\Delta_{i}$ 
%\end{equation*}
for some natural numbers $d_{i}$. 
By our third assumption of Theorem \ref{main thm: full cpt}, 
each $\Delta_{i}$ dominates $\Delta_{Q}'$, 
so we can choose a general point $y_{i}$ of $\Delta_{i}$ such that 
$f(y_{i})=y$ and  
the tangent map $T_{y_{i}}\Delta_{i}\to T_{y}\Delta_{Q}'$ is surjective. 
In some local coordinate around $y_{i}$, 
the projection $f\colon X \to {\AGcpt}$ is expressed as 
\begin{equation*}
(z'', z_{1}, \cdots , z_{N}, w_{1}, \cdots , w_{gn}) \mapsto 
(z'=(z'')^{d_{i}}, z_{1}, \cdots , z_{N}), 
\end{equation*}
with $\Delta_{i}$ defined by $z''=0$. 
We take a small neighborhood $\Delta_{y_{i}}$ of $y_{i}$ in $\Delta_{i}$ 
and let $U'_{\varepsilon,i}$ be 
the restriction of $U'_{\varepsilon}$ around $\Delta_{y_{i}}$.   
Then $U'_{\varepsilon,i}$ is the annulus bundle 
of radius $\varepsilon^{a/d_{i}} \leq |z''| \leq R^{a/d_{i}}$ around $\Delta_{y_{i}}$. 
%with radius measured by the normal parameter $w$. 

We can now state the version of the $L^{2/m}$ criterion we will apply. 

\begin{lemma}\label{prop: L2 criterion at I}
Let $\omega_{X}$ be an $m$-canonical form on $U'\cap {\XG}$ 
and $\omega_{I}$ be the pullback of $\omega_{X}$ to $U^{\circ}$. 

(1) Assume that for every positive number $\alpha>0$ the asymptotic estimate 
\begin{equation}\label{eqn: L2 criterion at I}
\int_{U_{\varepsilon}} ||\omega_{I} ||^{2/m} = o(\varepsilon^{-\alpha}) 
\qquad (\varepsilon \to 0)  
\end{equation}
holds. 
Then $\omega_{X}$ has at most pole of order $m$ along $U'\cap \Delta_{i}$ 
for every irreducible component $\Delta_{i}$ of $f^{-1}(\Delta_{Q}')$. 

(2) If 
$\int_{U_{\varepsilon}} ||\omega_{I} ||^{2/m} = O(1)$, 
then $\omega_{X}$ has at most pole of order $m-1$ 
along $U'\cap \Delta_{i}$ for every $\Delta_{i}$. 
\end{lemma}

\begin{proof}
(1) 
It is sufficient to look at the pole order of $\omega_{X}$ around $\Delta_{y_{i}}$. 
What has to be shown is that 
the asymptotic estimate \eqref{eqn: L2 criterion at I} for $\omega_{I}$ 
implies an asymptotic estimate for $\omega_{X}$ around $\Delta_{y_{i}}$ 
as in Proposition \ref{prop: L2 criterion} (1). 
Since $U_{\varepsilon,i}'$ is the annulus bundle 
of radius $[\varepsilon^{a/d_{i}}, R^{a/d_{i}}]$ around $\Delta_{y_{i}}$, 
the estimate in Proposition \ref{prop: L2 criterion} (1) for $\omega_{X}$ 
is equivalent to the estimate 
\begin{equation}\label{eqn: L2 criterion at I 2}
\int_{U_{\varepsilon,i}'} || \omega_{X} ||^{2/m} 
= o( (\varepsilon^{a/d_{i}})^{-2/m} ) 
= o( \varepsilon^{-2a/md_{i}} ) 
\qquad (\varepsilon \to 0).  
\end{equation}
Thus it suffices to show that 
\eqref{eqn: L2 criterion at I} implies \eqref{eqn: L2 criterion at I 2}. 

Since $U_{\varepsilon,i}' \subset U_{\varepsilon}'$ and 
$||\omega_{X}||^{2/m}$ is a nonnegative multiple of a volume form, 
we have 
\begin{equation*}
\int_{U_{\varepsilon,i}'} || \omega_{X} ||^{2/m} \leq 
\int_{U_{\varepsilon}'} || \omega_{X} ||^{2/m}. 
\end{equation*}
Pulling back $\omega_{X}$ to $U_{\varepsilon}$ 
by $U_{\varepsilon}\twoheadrightarrow U'_{\varepsilon}$, 
we have 
\begin{equation*}
\int_{U_{\varepsilon}'} || \omega_{X} ||^{2/m}  
% = {\rm deg}(\hat{p})^{-1} \int_{U_{\varepsilon}} || \omega_{I} ||^{2/m} 
\leq \int_{U_{\varepsilon}} || \omega_{I} ||^{2/m}. 
\end{equation*}
Then, if we substitute $\alpha=2a/md_{i}$ into \eqref{eqn: L2 criterion at I}, 
we obtain 
\begin{equation*}
\int_{U_{\varepsilon}} || \omega_{I} ||^{2/m} = o( \varepsilon^{-2a/md_{i}} ). 
\end{equation*}
This gives \eqref{eqn: L2 criterion at I 2}. 

The proof of (2) is similar, using 
Proposition \ref{prop: L2 criterion} (2) 
in place of Proposition \ref{prop: L2 criterion} (1). 
\end{proof}

%\begin{remark}
%By Theorem \ref{thm: toroidal} (2), $a$ equals to the ramification index of 
%$\mathcal{B}_{I}^{\Sigma}\to \mathcal{B}_{I}^{\Sigma}/{\GIZbar}$ at $\Delta_{Q}$. 
%Let $U(I)_{{\Z}}^{\star}={\UIQ}\cap \langle {\G}, -1 \rangle$. 
%Using Proposition \ref{prop: action of GIZ}, we can show that 
%$a=1$ if $Q/2\notin U(I)_{{\Z}}^{\star}$ and 
%$a=2$ if $Q/2\in U(I)_{{\Z}}^{\star}$.  
%\begin{equation*}
%a = 
%\begin{cases} 
%1 & Q/2\notin U(I)_{{\Z}}^{\star}, \\ 
%2 & Q/2\in U(I)_{{\Z}}^{\star}, 
%\end{cases}
%\end{equation*}
%like the classification of regular/irregular cusps in the case $g=1$. 
%We do not need this precise information for the proof of Theorem \ref{main thm: full cpt}. 
%\end{remark}

\subsection{Completion of the proof}

We translate the criterion in Lemma \ref{prop: L2 criterion at I} 
to that for the corresponding local modular forms. 
Combining the results so far, we obtain the following.  

\begin{proposition}\label{prop: F=O(1) L2}
Let $\omega_{X}$ and $\omega_{I}$ be as in Lemma \ref{prop: L2 criterion at I} 
and $k=(g+n+1)m$. 
Let $F$ be the section of $L^{\otimes k}$ over $V^{\circ}$ such that 
$\omega_{I}=f^{\ast}F$ under the isomorphism 
$K_{\mathcal{X}_{I}^{(n)}}^{\otimes m}\simeq f^{\ast}L^{\otimes k}$ 
of line bundles over $U^{\circ}$. 

(1) If $F$ extends to a holomorphic section of 
the extended line bundle $L^{\otimes k}$ over 
$V\subset \mathcal{B}_{I}^{\Sigma}$ (cf.~\S \ref{ssec: extend L}), 
then $\omega_{X}$ has at most pole of order $m$ along $U'\cap \Delta_{i}$ 
for every irreducible component $\Delta_{i}$ of $f^{-1}(\Delta_{Q}')$. 

(2) If furthermore $F$ vanishes at $\Delta_{Q}$, 
then $\omega_{X}$ has at most pole of order $m-1$ 
along $U'\cap \Delta_{i}$ for every $\Delta_{i}$. 
\end{proposition}
 
\begin{proof}
(1) 
We shall show that  the estimate \eqref{eqn: L2 criterion at I} holds, 
which by Proposition \ref{prop: L2 criterion modular} 
is equivalent to the estimate 
\begin{equation}\label{eqn: L2 modular to be shown}
\int_{V_{\varepsilon}} (F, F)_{k}^{1/m} {\volD} 
= o(\varepsilon^{-\alpha}). 
\end{equation}
We reduce this to Proposition \ref{prop: asymptotic Petersson norm} 
by passing from the $I$-cusp to an adjacent $0$-dimensional cusp. 
Choose a maximal isotropic sublattice $J$ of $\Lambda$ containing $I$. 
By Lemma \ref{lem: glue}, the projection ${\BI}\to \mathcal{B}_{J}$ 
extends to an etale map 
$\pi\colon \mathcal{B}_{I}^{\Sigma} \to \mathcal{B}_{J}^{\Sigma}$. 
We choose $\chi\in U(J)_{{\Z}}^{\vee}$ such that $\chi(Q)=1$ 
and put $q={\exp}(2\pi i \chi(\cdot ))$ 
on $\mathcal{B}_{J}\subset T_{J}$. 
Let $W_{\varepsilon}\subset \mathcal{B}_{J}$ be 
the annulus bundle around $\pi(\Delta_{x})$ of radius 
$\varepsilon \leq |q| \leq R$. 
As explained in Remark \ref{remark: independence on normal parameter}, 
the asymptotic behavior of the integral over $V_{\varepsilon}\simeq \pi(V_{\varepsilon})$ 
is equivalent to that over $W_{\varepsilon}$. 
Now we have 
\begin{equation*}
\int_{W_{\varepsilon}} (F, F)_{k}^{1/m} {\volD} 
= o(\varepsilon^{-\alpha}) 
\end{equation*}
by the first part of Proposition \ref{prop: asymptotic Petersson norm} with $\beta = 1/m$, 
which implies \eqref{eqn: L2 modular to be shown}. 

(2) 
Similarly, we are reduced to showing that 
\begin{equation*}
\int_{V_{\varepsilon}}(F, F)_{k}^{1/m} {\volD} 
= O(1),  
\end{equation*}
which follows from the second part of Proposition \ref{prop: asymptotic Petersson norm}. 
\end{proof}

We can now prove Theorem \ref{main thm: full cpt}. 
We first derive the isomorphism \eqref{main eqn: cpt}. %and \eqref{main eqn: cpt cusp form}. 
Let $\omega_{X}$ be an $m$-canonical form on ${\XG}$ and $F$ be the corresponding (global) modular form. 
Then $\omega_{I}$ corresponds to the restriction of $F$ to $V^{\circ}$. 
By Lemma \ref{lemma:Koecher toroidal}, 
$F$ extends holomorphically over $V$, 
so we can apply Proposition \ref{prop: F=O(1) L2} (1). 
We thus obtain the isomorphism \eqref{main eqn: cpt}. 
The assertion \eqref{main eqn: cpt cusp form} for cusp forms 
follows from the second part of Lemma \ref{lemma:Koecher toroidal}  
and Proposition \ref{prop: F=O(1) L2} (2). 

It remains to prove the last assertion of Theorem \ref{main thm: full cpt}. 
Let us recall the definition of klt singularities (\cite{K-M} \S 2.3). 

\begin{definition}
Let $X$ be a normal complex analytic variety 
and $\Delta$ be an effective ${\Q}$-Weil divisor such that $\lfloor \Delta \rfloor = 0$. 
The pair $(X, \Delta)$ is called \textit{Kawamata log terminal} 
if $K_{X}+\Delta$ is ${\Q}$-Cartier and 
for some (hence all) log resolution $\pi\colon Y\to X$ of $(X, \Delta)$, 
we have 
$K_{Y}=\pi^{\ast}(K_{X}+\Delta)+\sum_{i}a(E_i)E_i$ 
with $a(E_i)>-1$. 
If $\omega$ is a meromorphic $m$-canonical form on the regular locus of $X$ 
whose pole divisor satisfies $\leq m\Delta$, 
its pullback to $Y$ has at most pole of order $m-1$ along every component of the exceptional divisor. 
(This is the only property where we need the klt condition.) 
\end{definition}

Now let $X\supset {\XG}$ be as before and assume that $X$ is normal and compact.  
We write    
$\Delta_{U'}= f^{-1}(\Delta_{Q}')\cap U'$ 
for a general point $y$ of an irreducible component $\Delta_{Q}'$ of $\Delta_{A}$.  
What has to be shown is that if $(U', (1-m^{-1})\Delta_{U'})$ is klt 
for every component of $\Delta_{A}$, the map 
\begin{equation*}
\textit{S}_{(g+n+1)m}(\Gamma) \hookrightarrow H^0(K_{X}^{\otimes m}((m-1)\Delta_{X})) 
\end{equation*}
is surjective. 
Let $\omega_{X}$ be an element of $H^0(K_{X}^{\otimes m}((m-1)\Delta_{X}))$ 
and $F$ be the corresponding modular form of weight $k=(g+n+1)m$. 
In view of Lemma \ref{lemma:Koecher toroidal}, 
we want to show that $F$ vanishes at $\Delta_{Q,I}$ 
for every $I$ and positive definite ${\R}_{\geq0}Q\in \Sigma_I$. 
If we choose a maximal $J\supset I$, 
this is equivalent to the vanishing of $F$ at $\Delta_{Q,J}$. 
By the second part of Proposition \ref{prop: asymptotic Petersson norm} with $\beta=1/m$ 
(note that $km^{-1} \geq g+1$), 
it suffices to show that 
$\int_{W_{\varepsilon}}(F, F)_{k}^{1/m}{\volD}=O(1)$. 
Going back by the etale gluing 
$\mathcal{B}_{I}^{\Sigma_{I}}\to \mathcal{B}_{J}^{\Sigma_{J}}$, 
we are reduced to showing that 
$\int_{V^{\circ}}(F, F)_{k}^{1/m}{\volD}<\infty$. 
By Proposition \ref{prop: L2 criterion modular} 
this is translated to 
$\int_{U^{\circ}}||\omega_{I}||^{2/m}<\infty$, 
which in turn is equivalent to 
$\int_{U'-\Delta_{U'}}||\omega_{X}||^{2/m}<\infty$. 

We take a log resolution $(U'', \Delta_{U''})$ of $(U', \Delta_{U'})$ 
and let $E$ be its exceptional divisor. 
Then the above condition is rewritten as   
\begin{equation}\label{eqn: integral converge}
\int_{U''-E-\Delta_{U''}}||\omega_{X}||^{2/m}<\infty. 
\end{equation}
The divisor $E+\Delta_{U''}$ is simple normal crossing, 
and can be covered by finitely many local charts of $U''$ 
by the properness of $X\to{\AGcpt}$. 
Since the pole divisor of $\omega_{X}$ satisfies 
$\leq m(1-m^{-1})\Delta_{U'}$ by our assumption on $\omega_{X}$, 
the klt condition for $(U', (1-m^{-1})\Delta_{U'})$ implies that  
$\omega_{X}$ has at most pole of order $m-1$ along 
every component of $E+\Delta_{U''}$. 
Then the assertion \eqref{eqn: integral converge} 
follows from Proposition \ref{prop: L2 snc}. 
This completes the proof Theorem \ref{main thm: full cpt}. 
\qed

\begin{remark}
When $U'$ is smooth and $\Delta_{U'}$ is simple normal crossing, 
the pair $(U', (1-m^{-1})\Delta_{U'})$ is always klt. 
Hence  $\textit{S}_{(g+n+1)m}({\G})$ is isomorphic to $H^{0}(K_{X}^{\otimes m}((m-1)\Delta_{X}))$ 
for every $m$ when $X$ is smooth and $\Delta_{X}$ is simple normal crossing 
over general points of $\Delta_{A}$.  
\end{remark}

%%%%%
%%%%% Kodaira dimension  
%%%%%

\section{Proof of Theorem \ref{main thm: Kodaira dim}}\label{sec: Proof 1.3}

In this section we prove Theorem \ref{main thm: Kodaira dim}. 
In \S \ref{ssec: cano map} we prove the assertion (1). 
In \S \ref{ssec: Proof 1.3 (2)} we prove the assertion (2). 
For a ${\Q}$-Weil divisor $D$ on a normal compact complex analytic variety $X$, 
its Iitaka dimension $\kappa(D)=\kappa(X, D)$ is defined as 
the maximum of the dimension of the image of the rational map 
$\phi_{|mD|}\colon X \dashrightarrow {\proj}^{N}$ 
as $m$ runs (\cite{Ii}, \cite{Nakayama}). 
We write $K_{X}$ for the canonical divisor of $X$ as a Weil divisor 
or the corresponding rank $1$ reflexive sheaf on $X$.

\subsection{Canonical map}\label{ssec: cano map}

The canonical map of a smooth projective model of ${\XG}$ is described as follows. 

\begin{proposition}\label{prop: canonical map}
Let $X$ be a smooth projective model of ${\XG}$. 
The canonical map of $X$ factors through 
\begin{equation*}
X \sim {\XG} \twoheadrightarrow {\AG} \dashrightarrow {\proj}^{N} 
\end{equation*}
where ${\AG} \dashrightarrow {\proj}^{N}$ 
is the rational map defined by $\textit{S}_{g+n+1}({\G})$. 
\end{proposition}

When ${\AG} \dashrightarrow {\proj}^{N}$ is generically finite, 
this shows that the image of the canonical map of $X$ has dimension $g(g+1)/2$, 
which proves the assertion (1) of Theorem \ref{main thm: Kodaira dim}.  

\begin{proof}[(Proof of Proposition \ref{prop: canonical map})]
Recall %from \S \ref{ssec: proof 1.1} 
that the isomorphism 
$\textit{S}_{g+n+1}({\G})\simeq H^{0}(K_{X})$ 
is explicitly given as follows. 
We choose a maximal isotropic sublattice $J\subset \Lambda$ 
and let $s_{J}$ be the associated frame of $L$. 
Let $\omega_{{\D}}$ %=dz_{1}\wedge \cdots \wedge dz_{g(g+1)/2}$ 
be the canonical form on ${\D}$ corresponding to $s_{J}^{\otimes g+1}$ 
by $K_{{\D}}\simeq L^{\otimes g+1}$, 
and $\omega_{f}$ be the relative canonical form on $f\colon {\Xn}\to {\D}$ 
corresponding to $f^{\ast}s_{J}^{\otimes n}$ 
by $K_{f}\simeq f^{\ast}L^{\otimes n}$. 
We put $\omega_{0}=\omega_{f}\wedge f^{\ast}\omega_{{\D}}$. 
If $Fs_{J}^{\otimes g+n+1}$ is a cusp form of weight $g+n+1$, 
the pullback of the corresponding canonical form on $X$ to ${\Xn}$ is given by $(f^{\ast}F)\omega_{0}$. 
Thus, if $F_{i}s_{J}^{\otimes g+n+1}$, $0\leq i \leq N$, 
are basis of $\textit{S}_{g+n+1}({\G})$, 
then $(f^{\ast}F_{i})\omega_{0}$ are pullback of the corresponding basis of $H^{0}(K_{X})$ to ${\Xn}$. 
This shows that the canonical map of $X$ is given by 
$[ f^{\ast}F_{0} \colon \cdots \colon f^{\ast}F_{N}] : X \dashrightarrow {\proj}^{N}$. 
This factors through 
%\begin{equation*}
$X \stackrel{f}{\dashrightarrow} {\AG} \dashrightarrow {\proj}^{N}$,  
%\end{equation*}
where ${\AG} \dashrightarrow {\proj}^{N}$ is the rational map defined by $[F_{0}: \cdots : F_{N}]$. 
\end{proof}

\subsection{Proof of \eqref{main eqn: Kodaira dim}}\label{ssec: Proof 1.3 (2)}

Next we prove the assertion (2) of Theorem \ref{main thm: Kodaira dim}. 
Let $X\supset {\XG}$ be a normal complex analytic variety 
satisfying the conditions of Theorem \ref{main thm: full cpt}, 
and $f\colon X\to {\AGcpt}$ the extended morphism. 
We choose and fix a natural number $m$ such that 
$L^{\otimes m}$ can be defined as a line bundle over ${\AGcpt}$ 
in the sense of the last paragraph of \S \ref{ssec: extend L}. 
Since $K_{{\Xn}}^{\otimes m}\simeq f^{\ast}L^{\otimes k}$ over ${\Xn}$ 
where $k=(g+n+1)m$, 
then $K_{{\Xn}}^{\otimes m}$ also descends to a line bundle over ${\XG}$ 
which we denote by $K_{{\XG}}^{\otimes m}$ (by abuse of notation). 
Since ${\Xn}\to{\XG}$ is unramified in codimension $1$ 
by Lemma \ref{lem: unramify interior}, 
restriction of $K_{{\XG}}^{\otimes m}$ to the regular locus of ${\XG}$ 
is indeed isomorphic to the $m$-power of its canonical bundle. 
%Thus $K_{{\XG}}^{\otimes m}$ is linearly equivalent to $mK_{{\XG}}$ 
%where $K_{{\XG}}$ is the canonical divisor of ${\XG}$ as a Weil divisor. 
The isomorphism 
$K_{{\Xn}}^{\otimes m}\simeq f^{\ast}L^{\otimes k}$ 
over ${\Xn}$ descends to an isomorphism 
%\begin{equation*}
$K_{{\XG}}^{\otimes m}\simeq f^{\ast}L^{\otimes k}|_{{\XG}}$ 
%\end{equation*}
of line bundles over ${\XG}$.

\begin{proposition}\label{prop: sheaf version}
Let $X, m$ and $k=(g+n+1)m$ be as above. 
Then the isomorphism 
$f^{\ast}L^{\otimes k}|_{{\XG}} \simeq K_{{\XG}}^{\otimes m}$
of line bundles over ${\XG}$ 
extends to an injective homomorphism 
%\begin{equation*}\label{eqn: pole estimate sheaf}
$f^{\ast}L^{\otimes k} \hookrightarrow K_{X}^{\otimes m}(m\Delta_{X})$ 
%\end{equation*}
of sheaves on $X$. 
In particular, we have 
%\begin{equation*}
$mK_{X} \geq f^{\ast}(kL - m\Delta_{A})$. 
%\end{equation*}
\end{proposition}

\begin{proof}
We keep the notation of \S \ref{sec: proof 1.2}. 
%concerning the boundary. 
Let $y$ be a general point of an irreducible component $\Delta_{Q}'$ of $\Delta_{A}$. 
Let $F$ be a local frame of $L^{\otimes k}$ 
on a small neighborhood $V'$ of $y$ in ${\AGcpt}$. 
Then $f^{\ast}F$ is a local frame of $f^{\ast}L^{\otimes k}$ 
on the neighborhood $U'=f^{-1}(V')$ of $f^{-1}(y)$ in $X$. 
Let $\omega$ be the $m$-canonical form on $U' \cap {\XG}$ 
corresponding to the restriction of $f^{\ast}F$ to $U' \cap {\XG}$ 
via $K_{{\XG}}^{\otimes m}\simeq f^{\ast}L^{\otimes k}|_{{\XG}}$. 
By Proposition \ref{prop: F=O(1) L2} (1), 
$\omega$ has at most pole of order $m$ along $U' \cap \Delta_{i}$ 
for every irreducible component $\Delta_{i}$ of $f^{-1}(\Delta_{Q}')$. 
This shows that the isomorphism 
$f^{\ast}L^{\otimes k}|_{{\XG}} \simeq K_{{\XG}}^{\otimes m}$ 
at $U'\cap {\XG}$ 
extends to a sheaf homomorphism 
$f^{\ast}L^{\otimes k}|_{U'} \hookrightarrow K_{U'}^{\otimes m}(m\Delta_{U'})$ 
over $U'$. 
Since we obtain this for general points of every component of $\Delta_{A}$, 
our condition on $X\to {\AGcpt}$ ensures that 
we obtain a homomorphism  
$f^{\ast}L^{\otimes k} \hookrightarrow K_{X}^{\otimes m}(m\Delta_{X})$
outside a codimension $\geq2$ locus in $X$. 
By the normality of $X$, this extends over the whole $X$. 

As for the second assertion, 
we have 
%\begin{equation*}
$mK_{X} + m\Delta_{X} \geq f^{\ast}(kL)$ 
%\end{equation*}
by the first assertion. 
By %the normality of $X$ and ${\AGcpt}$ and by 
our condition on $X\to{\AGcpt}$, 
%after removing codimension $\geq2$ loci from both $X$ and ${\AGcpt}$,  
we can pullback $\Delta_{A}$ to a Weil divisor of $X$ 
whose support is $\Delta_{X}$. 
(Pullback $\Delta_{A}\cap {\AG}^{\Sigma}_{reg}$ as a Cartier divisor 
and take closure in $X$.) 
Then $f^{\ast}\Delta_{A}\geq \Delta_{X}$ and so  
%\begin{equation*}
$mK_{X} + m f^{\ast}\Delta_{A} \geq f^{\ast}(kL)$.  
%\end{equation*}
\end{proof}

%Theorem \ref{main thm: Kodaira dim} is deduced as follows. 

\begin{proof}[(Proof of (2) of Theorem \ref{main thm: Kodaira dim})]
By Proposition \ref{prop: sheaf version}, we have 
%\begin{eqnarray*}
%\kappa( X, K_{X}) 
%&\geq & 
%\kappa( X, \: f^{\ast}((g+n+1)L-\Delta_{A}))  \\ 
%&\geq &  
%\kappa({\AGcpt}, \: (g+n+1)L-\Delta_{A}). 
%\end{eqnarray*}
\begin{equation*}
\kappa( X, K_{X}) 
\: \geq \:  
\kappa( X, \, f^{\ast}((g+n+1)L-\Delta_{A}))  
\: \geq \: 
\kappa({\AGcpt}, \, (g+n+1)L-\Delta_{A}). 
\end{equation*}
On the other hand, since 
\begin{equation*}
H^{0}(X, K_{X}^{\otimes m'}) \subset 
H^{0}(X, K_{X}^{\otimes m'}(m'\Delta_{X})) \simeq  
M_{m'(g+n+1)}({\G}) 
\end{equation*} 
for every $m'$, 
we have 
%$\kappa(X, K_{X}) \leq g(g+1)/2$.   
%\begin{equation*}
$\kappa(X, K_{X})  \leq  
\kappa({\AGcpt}, (g+n+1)L)  =  g(g+1)/2$. 
%\end{equation*}
\end{proof}

%%%%%
%%%%% Singularities 
%%%%%

\section{Singularities}\label{sec: singularities} 

In this section, which is largely independent of the previous sections, 
we prove that ${\XG}$ has canonical singularities in most cases 
(Proposition \ref{prop: cano sing interior}).  
%Precisely: 
%when $g\geq4$ or 
%when $g=2, 3$ with ${\G}$ containing no element $\gamma$ of order $6$ whose eigenvalues on 
%$\Lambda_{{\C}}$ are $e(1/6), e(-1/6), 1, \cdots, 1$. 
%In \S \ref{ssec: prepare RT} we recall 
%
%In \S \ref{ssec: distribution} we prepare a lemma 
%on the distribution of eigenvalues of $\gamma$ on $V\subset \Lambda_{{\C}}$. 
%In \S \ref{ssec: singularity interior} 
%we prove our result on the singularities of ${\XG}$. 
%
Below, by a representation of a finite group $G$ over a field $K$, 
we mean a finite dimensional $K$-linear space 
equipped with a linear action of $G$.  
($K$ will be either ${\Q}$ or ${\C}$.) 
We write 
$e(\alpha)={\exp}(2\pi i \alpha)$ 
for $\alpha\in{\Q}/{\Z}$. 
%and especially $\zeta_{d}=e(1/d)$ for $d\in{\N}$. 

Let $W$ be a representation of a finite group $G$ over ${\C}$. 
The Reid--Shepherd-Barron--Tai criterion \cite{Re}, \cite{Ta} tells 
whether $W/G$ has canonical singularities 
in terms of the eigenvalues of elements of $G$. 
Let $\gamma\in G$ and $e(\alpha_{1}), \cdots, e(\alpha_{d})$ be 
the eigenvalues of $\gamma$ on $W$ where $d=\dim W$. 
We choose $\alpha_{i}\in{\Q}$ from $0\leq \alpha_{i} <1$. 
The \textit{Reid-Tai sum} %(also called the \textit{age}) 
of $\gamma$ is defined by 
\begin{equation*}
RT_{\gamma}(W) = \sum_{i=1}^{d} \alpha_{i}. 
\end{equation*}
The action of $\gamma$ on $W$ is called
\textit{quasi-reflection} (or \textit{pseudo-reflection}) 
if its eigenvalues are $1, \cdots, 1, \lambda$ with $\lambda\ne1$. 
  
\begin{theorem}[\cite{Re}, \cite{Ta}]
Assume that $G$ contains no quasi-reflection on $W$. 
Then $W/G$ has canonical singularities if and only if 
$RT_{\gamma}(W)\geq1$ for every element $\gamma\ne {\rm id}$ of $G$. 
\end{theorem}

%Note that by the cyclic reduction (\cite{Re}, \cite{Ta}), 
%$W/G$ has canonical singularities if and only if 
%$W/H$ has canonical singularities for every cyclic subgroup $H$ of $G$. 
%So when applying this RST criterion, we may eventually assume that $G$ is cyclic. 

We will apply this RST criterion for 
$W$ the tangent space $T_{p}{\Xn}$ of ${\Xn}$ at a point $p\in{\Xn}$ 
and  $G$ the stabilizer of $p$ in ${\G}$.

%%%Distribution lemma
\subsection{Distribution of eigenvalues}\label{ssec: distribution}

We first prepare a lemma on the distribution of eigenvalues. 
Let $G={\Z}/N$ be the standard cyclic group of order $N$. 
For $k\in{\Z}/N$ we write $\chi_{k/N}$ for the $1$-dimensional ${\C}$-representation of $G$ 
on which the standard generator $\bar{1}\in G$ acts by $e(k/N)$. 
Recall (\cite{Se} \S 13.1) that 
there is a unique faithful ${\Q}$-representation $\mathbb{V}_{N}$ of $G$ 
that is irreducible over ${\Q}$. 
%This can be defined as the kernel of 
%$\Phi_{n}(A) : {\Q}G \to {\Q}G$ 
%where $A\colon {\Q}G\to{\Q}G$ is the multiplication by $\bar{1}\in G$ 
%and $\Phi_{n}(x)$ the $n$-th cyclotomic polynomial. 
The complexification of $\mathbb{V}_{n}$ decomposes as 
\begin{equation*}
V_{N} := \mathbb{V}_{N} \otimes_{{\Q}} {\C} \simeq 
\bigoplus_{k\in({\Z}/N)^{\times}} \chi_{k/N}. 
\end{equation*}
For $d|N$, 
$\mathbb{V}_{d}$ is a representation of $G$ via the reduction ${\Z}/N\to{\Z}/d$. 
It is classical (\cite{Se} \S 13.1) that 
every ${\Q}$-representation of $G$ decomposes over ${\Q}$ into 
a direct sum of $\mathbb{V}_{d_{1}}, \cdots, \mathbb{V}_{d_{a}}$ 
for some $d_{1}, \cdots, d_{a}|N$. 
(We may have $d_{i}=d_{j}$ for $i\ne j$.) 

\begin{lemma}\label{lem: distribution}
Let $\Lambda_{{\Q}}$ be a representation of $G$ over ${\Q}$ and 
\begin{equation}\label{eqn: irr decomp /Q}
\Lambda_{{\Q}} = \bigoplus_{i=1}^{a} \mathbb{V}_{d_{i}}
\end{equation}
be an irreducible decomposition of $\Lambda_{{\Q}}$ over ${\Q}$. 
Assume that $G$ preserves a weight $1$ Hodge decomposition 
$\Lambda_{{\C}} = V \oplus \bar{V}$ of $\Lambda_{{\C}}$. 
%Then the following holds. 

(1) Let $d>2$. 
If $\mathbb{V}_{d}^{\oplus k}$ appears in \eqref{eqn: irr decomp /Q}, 
there is a sub $G$-representation $W_{d}$ of $V$ such that 
$W_{d}\oplus \bar{W_{d}}\simeq V_{d}^{\oplus k}$ 
as representations of $G$ over ${\C}$. 

(2) For $d=1, 2$ the multiplicity of $\mathbb{V}_{d}$ in \eqref{eqn: irr decomp /Q} is even, 
say $2k$, and $V$ contains a sub $G$-representation $V'$ isomorphic to $V_{d}^{\oplus k}$. 
\end{lemma}
 
\begin{proof}
(1) 
Let $d>2$. 
For a ${\C}$-representation $W$ of $G$ 
we write $\lambda(W)$ for the set of eigenvalues of $\bar{1}\in G$ 
\textit{counted with multiplicity}. 
We choose eigendecompositions of $V$ and $\bar{V}$ 
with respect to $\bar{1}\in G$: 
\begin{equation*}
V         = \bigoplus_{\lambda_{\alpha}\in \lambda(V)} {\C}v(\lambda_{\alpha}),  \qquad 
\bar{V} = \bigoplus_{\lambda_{\beta}'\in \lambda(\bar{V})} {\C}w(\lambda_{\beta}'), 
\end{equation*}
where 
$v(\lambda_{\alpha})\in V$ is a $\lambda_{\alpha}$-eigenvector 
and $w(\lambda_{\beta}')\in \bar{V}$ a $\lambda_{\beta}'$-eigenvector. 
We also fix a decomposition 
$\lambda(\Lambda_{{\C}}) = \lambda(V) \sqcup \lambda(\bar{V})$. 
Now, since $\Lambda_{{\Q}}$ contains $\mathbb{V}_{d}^{\oplus k}$, 
there exists \textit{some} embedding 
\begin{equation*}
\theta : \lambda(V_{d}^{\oplus k}) \hookrightarrow 
\lambda(\Lambda_{{\C}}) = \lambda(V) \sqcup \lambda(\bar{V}). 
\end{equation*}
($\mu\in \lambda(V_{d}^{\oplus k})$ and $\theta(\mu)\in \lambda(\Lambda_{{\C}})$ are the same number.) 
We put the elements of $\lambda(V_{d}^{\oplus k})$ 
by the order of their angle in $(0, 2\pi)$, say  
$\lambda(V_{d}^{\oplus k}) = \{ \mu_{1}, \cdots, \mu_{l} \}$. 
Then $\mu_{l+1-i}=\bar{\mu_{i}}$, 
and $\mu_{i}$ has angle in $(0, \pi)$ if $i\leq l/2$. 
We put 
\begin{equation*}
W_{d}^{+} := 
\bigoplus_{\begin{subarray}{c} i\leq l/2 \\ \theta(\mu_{i})\in\lambda(V) \end{subarray}} 
{\C}v(\theta(\mu_{i})), \qquad 
W_{d}^{-} := 
\bigoplus_{\begin{subarray}{c} j\leq l/2 \\ \theta(\mu_{j})\in\lambda(\bar{V}) \end{subarray}} 
{\C}\overline{w(\theta(\mu_{j}))}. 
\end{equation*}
Since $v(\lambda_{\alpha})\in V$ and 
$w(\lambda_{\beta}')\in \bar{V}$, 
we have $W_{d}^{+}, W_{d}^{-}\subset V$. 
We also have $W_{d}^{+}\cap W_{d}^{-}=\{ 0 \}$ because 
elements of $\lambda(W_{d}^{+})$ have angle in $(0, \pi)$ 
while those of $\lambda(W_{d}^{-})$ in $(\pi, 2\pi)$. 
Then we put 
%\begin{equation*}
$W_{d} = W_{d}^{+} \oplus W_{d}^{-}$. 
%\end{equation*}
Since
$\lambda(W_{d}^{+}) \sqcup \lambda(\overline{W_{d}^{-}}) = \{ \mu_{1}, \cdots , \mu_{l/2} \}$, 
by construction, 
we have 
\begin{equation*}
\lambda(W_{d}\oplus \bar{W_{d}}) = 
\lambda(W_{d}^{+})\sqcup \lambda(W_{d}^{-}) \sqcup 
\lambda(\overline{W_{d}^{+}}) \sqcup \lambda(\overline{W_{d}^{-}}) 
= \{ \mu_{1}, \cdots , \mu_{l} \}. 
\end{equation*}
Therefore $W_{d}\oplus \bar{W_{d}} \simeq V_{d}^{\oplus k}$ 
as abstract $G$-representations. 

(2) 
Let $d=1$ or $2$. 
Let $\mathbb{W}\subset \Lambda_{{\Q}}$ be the direct sum of 
all components $\mathbb{V}_{d_{i}}$ in \eqref{eqn: irr decomp /Q} such that $d_{i}=d$. 
Then $W=\mathbb{W}\otimes_{{\Q}} {\R}$ is the $(\pm1)$-eigenspace of $\bar{1}\in G$ on 
$\Lambda_{{\R}}$. 
Let $J:\Lambda_{{\R}}\to\Lambda_{{\R}}$ be the complex structure given by  
the Hodge decomposition $\Lambda_{{\C}}=V\oplus \bar{V}$. 
Since the $G$-action commutes with $J$, 
$J$ preserves $W$ and gives a complex structure on $W$. 
In particular, $\mathbb{W}$ has even dimension. 
If $W_{{\C}}=V'\oplus \bar{V}'$ is the Hodge decomposition given by $J|_{W}$, 
then $V'=V\cap W_{{\C}}$ is the $(\pm1)$-eigenspace of $\bar{1}\in G$ on $V$. 
%and has dimension $\dim W_{{\C}}/2$. 
\end{proof}

\subsection{Singularities of ${\XG}$}\label{ssec: singularity interior}

As before, let $\Lambda$ be a symplectic lattice of rank $2g>2$ and   
${\G}$ a finite-index subgroup of ${\rm Sp}(\Lambda)$. 
Our main result of \S \ref{sec: singularities} is the following.  

\begin{proposition}\label{prop: cano sing interior}
The Kuga variety ${\XG}$ has canonical singularities unless when 
$(g, n)=(2, 1), (3, 1), (2, 2)$ 
and ${\G}$ contains an element of order $6$ whose eigenvalues on $\Lambda_{{\C}}$ are 
$e(1/6), e(-1/6), 1, \cdots, 1$. 
\end{proposition}

\begin{proof}
Recall from \S \ref{ssec: quotient} 
that ${\XG}={\Xn}/{\G}$. 
Let $p=([V], x_{1}, \cdots , x_{n})$ be a point of ${\Xn}$ 
where $[V]\in{\D}$ and $x_{i}\in V^{\vee}/\Lambda$. 
It suffices to show that $T_{p}{\Xn}/\Gamma_{p}$ has canonical singularities 
where $\Gamma_{p}<{\G}$ is the stabilizer of $p$. 
By Lemma \ref{lem: unramify interior}, $\Gamma_{p}$ contains no quasi-reflection on $T_{p}{\Xn}$. 
Thus it suffices to show that 
$RT_{\gamma}(T_{p}{\Xn})\geq1$ for every $\gamma\ne {\rm id} \in\Gamma_{p}$. 
%Since ${\Xn}\to{\D}$ is ${\G}$-equivariant, 
%The $\Gamma_{p}$-representation $T_{p}{\Xn}$ decomposes as 
Since 
\begin{equation*}
T_{p}{\Xn} 
\:  \simeq  \: 
T_{x_{1}}(V^{\vee}/\Lambda) \oplus \cdots \oplus 
T_{x_{n}}(V^{\vee}/\Lambda) \oplus T_{[V]}{\D} 
\: \simeq \: 
(V^{\vee})^{\oplus n} \oplus {\sym}V^{\vee} 
\end{equation*}
as $\Gamma_{p}$-representations, 
we are reduced to the following calculation in linear algebra. 
(We rewrite $V^{\vee}$ as $V$, 
and $\Lambda_{{\Q}}^{\vee}$ as $\Lambda_{{\Q}}$.) 
\end{proof}

\begin{lemma}\label{lem: RT calculate}
Let $G=\langle \gamma \rangle$ be a finite cyclic group and 
$\Lambda_{{\Q}}$ a representation of $G$ over ${\Q}$ of dimension $2g>2$. 
Assume that $G$ preserves a Hodge decomposition 
$\Lambda_{{\C}} = V \oplus \bar{V}$  
and that $(n, \Lambda_{{\Q}})$ is neither of the following: 
\begin{itemize}
\item $n=1$, 
$\Lambda_{{\Q}}\simeq \mathbb{V}_6\oplus \mathbb{V}_1^{\oplus 2g-2}$ 
with $g=2, 3$. 
\item $n=2$, 
$\Lambda_{{\Q}}\simeq \mathbb{V}_6\oplus \mathbb{V}_1^{\oplus 2}$ ($g=2$). 
\end{itemize}
Then 
%\begin{equation*}
$RT_{\gamma}(V^{\oplus n}\oplus {\sym}V) \geq 1$. 
%\end{equation*}
\end{lemma}

\begin{proof}
As $G$-representation, one of the following cases occur: 
\begin{enumerate}
\item ${\LQ}\supset {\V}_d$, $\varphi(d)>2$; 
\item ${\LQ}\supset {\V}_3$; 
\item ${\LQ}\supset {\V}_4$; 
\item ${\LQ}\supset {\V}_6$; 
\item ${\LQ} = {\V}_1^{\oplus 2k} \oplus {\V}_2^{\oplus 2l}$. 
\end{enumerate}
We estimate the Reid-Tai sum case-by-case. %according to this classification. 
As in the proof of Lemma \ref{ssec: distribution}, 
for a $G$-representation $W$ over ${\C}$, 
we write $\lambda(W)$ for the set of eigenvalues of $\gamma$ counted with multiplicity. 
By associating $\alpha\in[0, 1)$ to $e(\alpha)$, 
we identify elements of $\lambda(W)$ with rational numbers in $[0, 1)$.  
We write $S^{2}V={\sym}V$. 

We first consider the case $n=1$. 
We write $RT=RT_{\gamma}(V\oplus S^2V)$. 
We show that $RT\geq 1$ unless 
$\Lambda_{{\Q}}\simeq \mathbb{V}_6\oplus \mathbb{V}_1^{\oplus 2g-2}$ 
with $g\leq 3$. 

(1) 
Let $W_d\subset V$ be the sub $G$-representation 
such that $W_{d}\oplus \bar{W}_{d}\simeq V_{d}$ 
as constructed in Lemma \ref{lem: distribution}. 
Firstly, if $\lambda(W_d)$ contains two elements $\lambda, \lambda'$ from $(1/2, 1)$, 
we have $RT> \lambda + \lambda'>1$. 
Secondly, suppose that $\lambda(W_d)$ contains exactly one element $\lambda$ from $(1/2, 1)$. 
Since $W_d\oplus \bar{W_d}\simeq V_d$, 
every element of $\lambda(V_d)\cap (0, 1/2)$ except $1-\lambda$ appears in $\lambda(W_d)$. 
Let $\lambda'$ be the maximal element of $\lambda(W_d)\cap (0, 1/2)$. 
When $\lambda'>1-\lambda$, we have $RT\geq \lambda+\lambda'>1$. 
When $\lambda'<1-\lambda$, we have 
$\lambda+\lambda'\in\lambda(S^2V)$ and $\lambda+\lambda' < 1$. 
Then $RT\geq \lambda+(\lambda+\lambda')>2\lambda>1$. 
Thirdly, if all elements of $\lambda(W_d)$ are contained in $(0, 1/2)$, 
we have $\lambda(W_d)=\lambda(V_d)\cap (0, 1/2)$ 
by $V_d\simeq W_d\oplus \bar{W_d}$. 
Let $\lambda$ be the maximal element of $\lambda(W_d)$. 
Then $\lambda > 1/4$. 
Since $\lambda(S^{2}V)$ contains $\lambda+\lambda'<1$ for every $\lambda'\in \lambda(V)$, 
we have $RT>(2+\varphi(d)/2)\lambda \geq 4\lambda >1$.

(2) 
By Lemma \ref{lem: distribution}, either 
$\chi_{1/3}\subset V$ or $\chi_{2/3}\subset V$. 
In the first case, we have $1/3\in\lambda(V)$ and $2/3\in\lambda(S^2V)$, 
so $RT\geq 1/3+2/3=1$. 
In the second case, we have $2/3\in\lambda(V)$ and $1/3\in\lambda(S^2V)$, 
so again $RT\geq 1$.

(3) 
Since $g>1$, we have ${\LQ}\ne{\V}_4$. 
In view of the cases (1), (2), we only need to consider the case 
${\V}_4\oplus{\V}_d\subset {\LQ}$ with $d=1, 2, 4, 6$. 
When ${\V}_4^{\oplus2}\subset {\LQ}$, $\lambda(V)$ contains two elements from 
$\{ 1/4, 1/4, 3/4, 3/4 \}$ by Lemma \ref{lem: distribution}, 
so $\lambda(S^2V)$ contains two $1/2$  
and hence $RT>1$. 
When ${\V}_4\oplus{\V}_6\subset {\LQ}$, 
$\lambda(V)$ contains 
$\{ 1/4 \, \textrm{or} \, 3/4, 1/6 \, \textrm{or} \, 5/6 \}$ 
by Lemma \ref{lem: distribution}. 
Then $\lambda(S^2V)$ contains $\{ 1/2, 1/3 \, \textrm{or} \, 2/3 \}$. 
Hence $RT>1$. 
Similarly, when ${\V}_4\oplus{\V}_2\subset {\LQ}$, 
$\lambda(V)$ contains $\{ 1/2, 1/4 \, \textrm{or} \, 3/4 \}$, and so 
$\lambda(S^2V)$ contains $\{ 3/4 \, \textrm{or} \, 1/4, 1/2 \}$, 
which implies $RT>1$. 
Finally, when ${\V}_4\oplus{\V}_1\subset {\LQ}$, 
$\lambda(V)$ contains $\{ 0, 1/4 \, \textrm{or} \, 3/4 \}$, and so 
$\lambda(S^2V)$ contains $\{ 1/2, 1/4 \, \textrm{or} \, 3/4 \}$. 
This proves $RT\geq1$.

(4) 
In view of the cases (1) -- (3), 
we only need to cover the cases 
\begin{equation*}
{\LQ} \supset {\V}_6\oplus{\V}_6, \quad 
{\LQ} \supset {\V}_6\oplus{\V}_2, \quad 
{\LQ} = {\V}_6\oplus{\V}_1^{\oplus 2g-2} \: \: (g \geq 4). 
\end{equation*}
When ${\V}_6^{\oplus2}\subset{\LQ}$, 
$\lambda(V)$ contains two elements from 
$\{ 1/6, 1/6, 5/6, 5/6 \}$ by Lemma \ref{lem: distribution}, 
so $\lambda(S^2V)$ contains two elements from 
$\{ 1/3, 1/3, 2/3, 2/3 \}$. 
It follows that $RT\geq1$. 
When ${\V}_6\oplus{\V}_2\subset{\LQ}$, 
$\lambda(V)$ contains 
$\{ 1/2,  1/6 \, \textrm{or} \, 5/6 \}$, 
so $\lambda(S^2V)$ contains $1/3$ or $2/3$. 
Thus $RT\geq1$. 
Finally, when ${\LQ} = {\V}_6\oplus{\V}_1^{\oplus 2g-2}$, 
we have 
$V\simeq \chi_{\pm1/6}\oplus \chi_1^{\oplus g-1}$. 
If $V\simeq \chi_{5/6}\oplus \chi_1^{\oplus g-1}$, 
then $\chi_{5/6}\subset S^2V$, and so $RT>1$. 
If $V\simeq \chi_{1/6}\oplus \chi_1^{\oplus g-1}$, we have 
%$\lambda(V)=\{ 1/6, 0, \cdots , 0 \}$ and 
%$\lambda({\rm Sym}^2V)=\{ 1/3, 1/6, \cdots, 1/6, 0, \cdots , 0 \}$ 
%where $1/6$ has multiplicity $g-1$. 
%Hence 
$RT=1/3+g/6\geq 1$ 
by $g\geq4$.

(5) 
In this case we have 
$V\simeq \chi_{1}^{\oplus k}\oplus \chi_{1/2}^{\oplus l}$ by Lemma \ref{lem: distribution}. 
When $l\geq2$, we have 
$RT\geq l/2 \geq1$. 
When $l=1$, we have $k\geq1$ by $g\geq2$. 
Then $1/2\in\lambda(V)$ and $1/2\in \lambda(S^2V)$, 
so $RT\geq1$. 
This finishes the proof for the case $n=1$. 

Next let $n\geq2$. 
Since 
%\begin{equation*}
$RT_{\gamma}(V^{\oplus n}\oplus S^2V) 
\geq RT_{\gamma}(V\oplus S^2V)$, 
%\end{equation*}
we only need to consider the case 
${\LQ} = {\V}_6\oplus{\V}_1^{\oplus 2g-2}$ 
and $V\simeq \chi_{1/6}\oplus \chi_{1}^{\oplus g-1}$ 
with $g=2, 3$ by the above proof for the case $n=1$. 
In this case the Reid-Tai sum is 
$1/3+(g+n-1)/6$, 
which is smaller than $1$ only when $(g, n)=(2, 2)$. 
This completes the proof of Lemma \ref{lem: RT calculate} 
and hence of Proposition \ref{prop: cano sing interior}. 
\end{proof}

%%%%%%% Reference %%%%%%%%%%%%%%%%%%%%%%%%%%%%%

\end{document}